\documentclass[11pt,a4paper,reqno]{amsart}
\usepackage{amssymb,amsmath,amsthm}
\usepackage[utf8]{inputenc}
\usepackage[T1]{fontenc}
\usepackage{enumerate}
\usepackage[all]{xy}
\usepackage{fullpage}
\usepackage{comment}
\usepackage{array}
\usepackage{longtable}
\usepackage{stmaryrd}
\usepackage{mathrsfs} 
\usepackage{xcolor}
\usepackage{mathtools}

\def\gcd{\mathrm{gcd}}
\def\deg{\mathrm{deg}}

\def\int{\mathrm{Int}}

\def\dim{\mathrm{dim}}

\usepackage{hyperref}
\hypersetup{
    colorlinks=true,
    linkcolor=blue,
    filecolor=red,  
    citecolor=green,
    urlcolor=cyan,
    pdftitle={Hausdorff Dimension of Singular Vectors in Function Fields},
    pdfpagemode=FullScreen,
    }

\usepackage{cleveref}
\crefformat{section}{\S#2#1#3}
\crefformat{subsection}{\S#2#1#3}
\crefformat{subsubsection}{\S#2#1#3}
\usepackage{enumitem}
\usepackage{tikz}

\newtheorem{theorem}{Theorem}[section]
\newtheorem{lemma}[theorem]{Lemma}
\newtheorem{corollary}[theorem]{Corollary}
\newtheorem{proposition}[theorem]{Proposition}

\theoremstyle{definition}
\newtheorem{definition}[theorem]{Definition}

\theoremstyle{remark}
\newtheorem{remark}[theorem]{Remark}

\makeatletter
\newcommand{\subalign}[1]{%
  \vcenter{%
    \Let@ \restore@math@cr \default@tag
    \baselineskip\fontdimen10 \scriptfont\tw@
    \advance\baselineskip\fontdimen12 \scriptfont\tw@
    \lineskip\thr@@\fontdimen8 \scriptfont\thr@@
    \lineskiplimit\lineskip
    \ialign{\hfil$\m@th\scriptstyle##$&$\m@th\scriptstyle{}##$\hfil\crcr
      #1\crcr
    }%
  }%
}

\makeatother

\numberwithin{equation}{section}

\title{Hausdorff Dimension of Singular Vectors in Function Fields}
\author{Noy Soffer Aranov}
\address{Noy ~Soffer ~Aranov. Department of Mathematics, Technion, Haifa, Israel}
\email{noyso@campus.technion.ac.il}

\author{Taehyeong Kim}
\address{Taehyeong ~Kim. The Einstein Institute of Mathematics, Edmond J. Safra Campus, Givat Ram\\
The Hebrew University of Jerusalem, Jerusalem, 91904, Israel}
\email{taehyeong.kim@mail.huji.ac.il}

\begin{document}

\def\thefootnote{}
\footnote{The first named author is supported by the ERC grant "Dynamics on Homogeneous Spaces" (no. 754475).} 
\footnote{The second named author is supported by the ERC grant HomDyn, ID 833423.}

\begin{abstract}
    We compute the Hausdorff dimension of the set of singular vectors in function fields and bound the Hausdorff dimension of the set of $\varepsilon$-Dirichlet improvable vectors in this setting. This is a function field analogue of the results of Cheung and Chevallier [Duke Math. J. \textbf{165} (2016), 2273--2329]. 
\end{abstract}
\maketitle

\section{Introduction}
A fundamental theorem in Diophantine approximation is Dirichlet's Theorem which says that for all $\boldsymbol{\theta}=(\theta_1,\dots,\theta_d)\in\mathbb{R}^d$ and $T>1$ there exists $(\mathbf{a},b)=(a_1,\dots,a_d,b)\in \mathbb{Z}^d\times\mathbb{Z}$ such that 
\[
\max_{1\leq i\leq d}\vert b\theta_i-a_i\vert<\frac{1}{T^{\frac{1}{d}}}\qquad\text{and}\qquad 0<\vert b\vert < T.
\]
It is natural to ask whether Dirichlet's Theorem can be improved, and this question leads to the following definition. Given $\varepsilon>0$, a vector $\boldsymbol{\theta}\in\mathbb{R}^d$ is said to be \textit{$\varepsilon$-Dirichlet improvable} if there is $T_0>1$ such that for all $T>T_0$ there exists $(\mathbf{a},b)\in\mathbb{Z}^d\times\mathbb{Z}$ such that 
\[
\max_{1\leq i\leq d}\vert b\theta_i-a_i\vert<\frac{\varepsilon}{T^{\frac{1}{d}}}\qquad\text{and}\qquad 0<\vert b\vert < T.
\]
We say that a vector $\boldsymbol{\theta}\in\mathbb{R}^d$ is \textit{singular} if it is $\varepsilon$-Dirichlet improvable for every $\varepsilon>0$. Denote by $\mathbf{DI}_d(\varepsilon)$ and $\mathbf{Sing}_d$ the set of $\varepsilon$-Dirichlet improvable vectors and singular vectors in $\mathbb{R}^d$, respectively. 

Khintchine \cite{K} originally introduced singular vectors and showed that $\mathbf{Sing}_d$ has Lebesgue measure zero. Davenport and Schmidt \cite{DS} proved that $\mathbf{DI}_d(\varepsilon)$ has Lebesgue measure zero for any $0<\varepsilon<1$. It is worth noting that $\mathbf{Sing}_1 = \mathbb{Q}$, but if $d\geq 2$, the singular set is nontrivial. The computation of the Hausdorff dimension of $\mathbf{Sing}_d$ with $d\geq 2$ has been a challenge until the first breakthrough by Cheung \cite{C11}, who proved that the Hausdorff dimension of $\mathbf{Sing}_2$ is $4/3$ and obtained a certain bound for the Hausdorff dimension of $\mathbf{DI}_2(\varepsilon)$. This was extended by Cheung and Chevallier \cite{CC16} to higher dimensions as follows:
\begin{theorem}\cite{CC16}\label{Thm_CC}
Let $d\geq 2$. Then
\begin{enumerate}
    \item the Hausdorff dimension of $\mathbf{Sing}_d$ is $\frac{d^2}{d+1}$;
    \item\label{CCdim_HDI} for any $t>d$, there is a constant $C>0$ such that for all small enough $\varepsilon>0$,
    \[
     \frac{d^2}{d+1}+\varepsilon^t\leq \dim_H\mathbf{DI}_d(\varepsilon)\leq \frac{d^2}{d+1}+C\varepsilon^{\frac{d}{2}}.
    \]
\end{enumerate}
\end{theorem}
Here and hereafter, $\dim_H$ stands for the Hausdorff dimension.
We remark that this result improves the bounds in \cite{B1,B2,R} and refer to \cite{KKLM,DFSU} for the Hausdorff dimension of singular matrices.

In this paper, we shall prove a function field analogue of Theorem \ref{Thm_CC}. We first introduce the function field setting. Let $p$ be a prime number, $q$ be a power of $p$, and let $\mathcal{R}=\mathbb{F}_q[x]$ be the ring of polynomials over $\mathbb{F}_q$. Let $\mathcal{K}=\mathbb{F}_q(x)$ be the field of rational functions over $\mathbb{F}_q$. We define an absolute value on $\mathcal{R}$ by $\vert f\vert=q^{\deg(f)}$ and extend it to an absolute value on $\mathcal{K}$ by $\left|\frac{f}{g}\right|=q^{\deg(f)-\deg(g)}$. Then the topological completion of $\mathcal{K}$ with respect to the metric $d(f,g)=\vert f-g\vert$ is the field of Laurent series $\widetilde{\mathcal{K}}$ defined by
$$\widetilde{\mathcal{K}}\:= \mathbb{F}_q\left(\left(x^{-1}\right)\right)=\Bigg\{\sum_{n=-N}^{\infty} a_n x^{-n}:a_n\in \mathbb{F}_q,N\in \mathbb{Z}\Bigg\}.$$
We equip $\widetilde{\mathcal{K}}^d$ with the supremum norm 
$\Vert \mathbf{v}\Vert=\max_{i=1,\dots, d}\vert v_i\vert$ for $\mathbf{v}=(v_1,\dots,v_d)\in \widetilde{\mathcal{K}}$ and consider the Hausdorff dimension of subsets of $\widetilde{\mathcal{K}}^d$ with respect to this supremum norm.

We remark that the ring of polynomials $\mathcal{R}$ is discrete in $\widetilde{\mathcal{K}}$ and plays a similar role to the ring of integers $\mathbb{Z}$ in $\mathbb{R}$.
In this point of view, there have been numerous developments
on Diophantine approximation in the function field setting, see for instance \cite{dM,Las}.
In particular, Dirichlet's Theorem also holds (see e.g. \cite[Chapter IV]{dM}). 
Hence, it makes sense to consider the following definition as in the real case: 

Given $\varepsilon>0$, we say that $\boldsymbol{\theta}\in \widetilde{\mathcal{K}}^d$ is \textit{$\varepsilon$-Dirichlet improvable} if there is $T_0>1$ such that for all $T>T_0$ there exists $(\mathbf{a},b)\in\mathcal{R}^d\times\mathcal{R}$ such that 
\[
\| b\boldsymbol{\theta}-\mathbf{a}\|<\frac{\varepsilon}{T^{\frac{1}{d}}}\qquad\text{and}\qquad 0<|b|< T.
\]
We say that a vector $\boldsymbol{\theta}\in\widetilde{\mathcal{K}}^d$ is \textit{singular} if it is $\varepsilon$-Dirichlet improvable for every $\varepsilon>0$. Denote by $\mathbf{DI}_d(\varepsilon)$ and $\mathbf{Sing}_d$ the set of $\varepsilon$-Dirichlet improvable vectors and singular vectors in $\widetilde{\mathcal{K}}^d$, respectively.

Compared to intensive studies in the real setting as mentioned above, there has been less attention to singular sets over function fields. In particular, it was proved in \cite[Theorem 3.6]{GaGh17} that $\mathbf{DI}_d(\varepsilon)$ has zero Haar measure on $\widetilde{\mathcal{K}}^d$ for every $\varepsilon$ small enough. Hence, $\mathbf{Sing}_d$ has also zero Haar measure.
Recently, it has been proved by Bang, the second named author, and Lim \cite{BKL} that the Hausdorff dimension of $\mathbf{Sing}_d$ is at most $\frac{d^2}{d+1}$. We prove that their upper bound is sharp, and moreover we bound the Hausdorff dimension of $\mathbf{DI}_d(\varepsilon)$.
\begin{theorem}[Corollaries \ref{lem:DIUpp}+\ref{Cor_lowerbound}+\ref{Cor_singlowerbound}]\label{Thm_Main} 
Let $d\geq 2$. Then
\begin{enumerate}
    \item the Hausdorff dimension of $\mathbf{Sing}_d$ is $\frac{d^2}{d+1}$;
    \item for any $\varepsilon>0$, we have
    \[\begin{split}
     &\frac{d^2}{d+1}+\frac{d}{d+1}\log_q \left(1+\frac{(q-1)^2}{q^{2d+\frac{d^2}{d-1}}}\left(\frac{q-1}{q}-\varepsilon^{d-1}\right)\varepsilon^{d}\right)\\ &\leq \dim_H \mathbf{DI}_d(\varepsilon) 
      \leq \frac{d^2}{d+1}+\frac{d}{d+1}\log_q\left(1+\sqrt{(q-1)^2(q+1)^dq^{2d} \varepsilon^d}\right).
    \end{split}\]
\end{enumerate}
\end{theorem}
\begin{remark} We remark several things about the second argument in Theorem \ref{Thm_Main}.
    \begin{enumerate}
        \item Since $\log_q(1+y)=y/\log q +o(y)$ as $y\to 0$, Theorem \ref{Thm_Main} has the same order of magnitude as in Theorem \ref{Thm_CC} for all small enough $\varepsilon$. (In fact, the lower bound of Theorem \ref{Thm_Main} is slightly better than $\epsilon^t$ in Theorem \ref{Thm_CC}.)
        \item Note that the lower bound is increasing with $0<\varepsilon\leq \left(\frac{d(q-1)}{(2d-1)q}\right)^{\frac{1}{d-1}}$ and the upper bound is nontrivial for $0<\varepsilon\leq (q+1)^{-1}q^{-2}$, otherwise the upper bound is greater than $d$.  Therefore, the nontrivial regions of $\varepsilon$ are $\left(0,\left(\frac{d(q-1)}{(2d-1)q}\right)^{\frac{1}{d-1}}\right]$ and $(0,(q+1)^{-1}q^{-2}]$ for the lower bound and the upper bound, respectively. 
        \item While Theorem \ref{Thm_CC} only works for $\varepsilon$ small enough, our result works for relatively large $\varepsilon$. We are able to obtain this stronger result by the ultrametric nature which will be explained at the end of the introduction. In both real and function field cases, it is totally open to determine/bound reasonably the Hausdorff dimension of $\mathbf{DI}_d(\varepsilon)$ for any $\varepsilon>0$. See \cite{KM, KSY} for $\varepsilon<1$ sufficiently near $1$.
        \item We note that in \cite[Section 7]{CC16}, Cheung and Chevallier proved that when $d=2$, the exponent $t$ in Theorem \ref{Thm_CC} can be taken to be greater than $1$, instead of greater than $2$. In order to do so, they defined a more refined fractal structure, and used a claim about chord length in balls in $\mathbb{R}^2$ \cite[Lemma 7.7]{CC16} to show that the structure they defined is indeed nested. Since the geometry in the function field setting is non-Euclidean, similar results cannot be achieved immediately and new ideas would be needed.
    \end{enumerate} 
\end{remark}
One may think of the case that the base field is infinite. In this case, there have been several developments in Diophantine approximation as follows: Roth theorem \cite{Uch} and Schmidt subspace theorem \cite{Rat} over base fields of characteristic zero; a disproof of Littlewood conjecture \cite{DL} when the base field is infinite (see also \cite{AB}).
However, the polynomial ring over an infinite base field behaves differently from integers. In particular, the number of polynomials within a certain radius is infinite. Since our proof relies heavily on counting polynomials, we cannot handle this case.

We will now discuss a dynamical interpretation of our main theorem. Let $\operatorname{SL}_{d+1}(\widetilde{\mathcal{K}})$ be the group of invertible $(d+1) \times (d+1)$ matrices $g$ over $\widetilde{\mathcal{K}}$ with $\vert \det(g)\vert=1$ and let $\operatorname{SL}_{d+1}(\mathcal{R})$ be the subgroup of $\operatorname{SL}_{d+1}(\widetilde{\mathcal{K}})$ with entries in $\mathcal{R}$. Then, the homogeneous space $\mathcal{L}_{d+1}=\operatorname{SL}_{d+1}(\widetilde{\mathcal{K}})/\operatorname{SL}_{d+1}(\mathcal{R})$ is identified with the space of unimodular lattices in $\widetilde{\mathcal{K}}^{d+1}$ under the identification $$g\operatorname{SL}_{d+1}(\mathcal{R})\mapsto g\mathcal{R}^{d+1}.$$
There is a well-known dynamical interpretation of singular vectors. For $\boldsymbol{\theta}\in \widetilde{\mathcal{K}}^d$, consider the lattice $\Delta_{\boldsymbol{\theta}}=\begin{pmatrix}
    \mathbf{I}_d&\boldsymbol{\theta}\\
    0&1
\end{pmatrix}\mathcal{R}^{d+1}$ and the diagonal flow $g_t=\operatorname{diag}\{x^{t},\dots,x^{t},x^{-dt}\}$. Then, $\boldsymbol{\theta}$ is singular if and only if the orbit $\{g_t\Delta_{\boldsymbol{\theta}}:t\in\mathbb{N}\}$ is divergent, that is eventually leaves every compact set in $\mathcal{L}_{d+1}$. See for instance \cite[Lemma 2.1]{BKL}.

This relationship was originally studied by Dani \cite{D} in the real setting. Divergent trajectories for diagonal flows have been well studied in the real setting (see for example \cite{KW,W1,W2,T,ST,DFSU,KKLM}).

Theorem \ref{Thm_Main} implies the following dynamical result as in \cite[Corollary 1.2]{CC16}.
\begin{corollary}
    For $d\geq 2$, the Hausdorff dimension of the set
    \[
    \{\Delta\in\mathcal{L}_{d+1}:\text{ the orbit }\{g_t\Delta:t\in\mathbb{N}\}\text{ is divergent}\}
    \]
    is $\dim_H \operatorname{SL}_{d+1}(\widetilde{\mathcal{K}})-\frac{d}{d+1}$.
\end{corollary}
\begin{proof}
    The same proof verbatim of \cite[Corollary 1.2]{CC16} works in our setting. See also the proof of \cite[Proposition 6.5]{BKL}. 
\end{proof}

Finally, we discuss how we are able to obtain stronger results than Theorem \ref{Thm_CC}. We basically follow the proof of \cite{CC16}, but the ultrametric nature enables us to obtain better results. By the ultrametric nature, we mean the following lemma.
\begin{lemma}[Ultrametric inequality]
    \label{umIneq}
    For every $\mathbf{v},\mathbf{w}\in \widetilde{\mathcal{K}}^d$, we have
    \begin{equation}
    \label{eqn:ultrametric}
        \| \mathbf{v}+\mathbf{w}\|\leq \max\{\|\mathbf{v}\|,\|\mathbf{w}\|\}.
    \end{equation}
    Moreover, the inequality in (\ref{eqn:ultrametric}) is an equality if $\|\mathbf{v}\|\neq \|\mathbf{w}\|$.
\end{lemma}
There are two main reasons for the difference between Theorems \ref{Thm_CC} and \ref{Thm_Main}. The first reason is about the characterization of $\varepsilon$-Dirichlet improvability in terms of the length of the shortest vector in the \textit{Farey lattice} given in Definition \ref{Def_Fareylattice} with respect to best approximations. This corresponds to \cite[Corollary 4.4]{CC16}. In the real case, the characterization works up to multiplication by $1/2$ and by $2$, but the ultrametric nature allows us to obtain the exact equivalence as in Corollary \ref{Cor_SingChar}. The second reason comes from Minkowski's second theorem (Theorem \ref{Mink2nd}). While Minkowski's second theorem on $\mathbb{R}^d$ works with $2^d/d!$ and $2^d$ multiplications (see \cite[Equation (M)]{CC16} and \cite{Cas59}), we have the equality over $\widetilde{\mathcal{K}}^d$ as in Theorem \ref{Mink2nd}. These two reasons ultimately ensure that the fractal structures we build work well for all $0<\varepsilon<1$, unlike \cite{CC16}.

\vspace{.3cm}
The article is organized as follows. In \cref{Sec:pre}, we recall some basic estimates in number theory and geometry of numbers which will be used later. In \cref{sec:SSCov}, we recall the definition of self-similar structures and some related dimension results given in \cite{CC16}.
In \cref{sec:bestApprox}, we characterize $\varepsilon$-Dirichlet improvability in terms of best approximations, define the Farey lattice, and connect it with best approximations. In \cref{sec:UpBnd} and \cref{sec:LowBnd}, we define the fractal structures using the Farey lattices and compute the upper and lower bounds of the Hausdorff dimensions of $\mathbf{DI}_d(\varepsilon)$ and $\mathbf{Sing}_d$, respectively. 

\vspace{5mm}
\noindent\textbf{Acknowledgments}. We would like to thank the anonymous referees for carefully reading this
article and providing helpful comments.
\section{Preliminaries}\label{Sec:pre}
\subsection{Basic number theory}\label{subsec:BasicNT}
In this subsection, we gather some notation and facts about basic number theory in function fields following \cite{Ros}. 

Recall that $\mathbb{F}_q$ is a finite field with $q$ elements for a prime power $q$ and $\mathcal{R}=\mathbb{F}_q[x]$ is the ring of polynomials over $\mathbb{F}_q$. A non-zero polynomial $f\in\mathcal{R}$ is called a \textit{monic polynomial} if the leading coefficient is equal to 1. Monic polynomials play the role of positive integers. For $f,g \in \mathcal{R}$, the \textit{greatest common divisor} of $f$ and $g$, denoted $\gcd(f,g)$, is the monic polynomial of largest degree which divides both $f$ and $g$. Define the Euler totient function $\phi(f)$ of $f\in\mathcal{R}$ to be the number of non-zero polynomials of degree less than $\deg(f)$ and relatively prime to $f$, that is,
$$\phi(f)=\#\{g\in\mathcal{R}\setminus\{0\}:\deg(g)<\deg(f)\ \text{ and } \gcd(f,g)=1\}.$$
Consider the divisor function $D_1(f)$ of $f\in \mathcal{R}$ defined by $D_1(f)=\sum_{g|f}|g|$, where $\sum_{g|f}$ stands for the summation over all monic divisors of $f$.
We will need the following estimates:
\begin{lemma}{\cite[Proposition 2.7]{Ros}}
\label{lem:phiEval}
For every integer $\ell\geq 1$, we have 
$$\sum_{\deg(f)=\ell}\phi(f)=\frac{(q-1)^2}{q}q^{2\ell}.$$
\end{lemma}
\begin{proof}
    It follows from \cite[Proposition 2.7]{Ros} that
    $$\sum_{\substack{\deg(f)=\ell \\ f\ \text{monic}}}\phi(f)=\frac{q-1}{q}q^{2\ell}.$$
    Given monic polynomial $f\in\mathcal{R}$, if $g\in \mathbb{F}_q^* f$, then $\phi(g)=\phi(f)$. Hence the lemma follows.
\end{proof}
\begin{lemma}
    \label{lem:D1Eval}
    For every integer $\ell\geq 1$, we have
    $$\sum_{\deg(f)=\ell}\phi(f)D_1(f)\leq (q-1)q^{3\ell}.$$
\end{lemma}
\begin{proof}
    Given a non-zero polynomial $f\in \mathcal{R}$, let $f=\alpha P_1^{e_1}P_2^{e_2}\cdots P_t^{e_t}$ be its decomposition to irreducible components, where $\alpha \in \mathbb{F}_q^*$, $P_i$ are distinct monic irreducible polynomials, and $e_i$ are non-negative integers.
    It follows from \cite[Propositions 1.7 and 2.4]{Ros} that
    \[
    \phi(f)=\prod_{i=1}^t \left(|P_i|^{e_i}-|P_i|^{e_i-1}\right)\quad\text{and}\quad D_1(f)=\prod_{i=1}^t \frac{|P_i|^{e_i+1}-1}{|P_i|-1}.
    \]
    Hence, we have
    \[
    \phi(f)D_1(f)=|f|\prod_{i=1}^t\frac{\vert P_i\vert^{e_i+1}-1}{\vert P_i\vert}
    =|f|^2 \prod_{i=1}^t\left(1-\frac{1}{|P_i|^{e_i+1}}\right) \leq |f|^2.
    \]
    Therefore, 
    \[
    \sum_{\deg(f)=\ell} \phi(f)D_1(f)\leq \sum_{\deg(f)=\ell}q^{2\ell}= (q-1)q^{3\ell}.
    \]
\end{proof}

\subsection{Geometry of numbers}
\label{subsec:GON}
An important concept in the geometry of numbers is successive minima. Minkowski's second theorem connects between the successive minima and the determinant/covolume of a lattice. The function field version of this theorem was proved by Mahler \cite{Mah}. In order to state it, we define the determinant of a lattice.
\begin{definition}
    For a lattice $\Lambda=g\mathcal{R}^d$ in $\widetilde{\mathcal{K}}^d$ with $g\in\mathrm{GL}_d(\widetilde{\mathcal{K}})$, we define its determinant by $$\det(\Lambda)=\vert \det(g)\vert.$$
\end{definition}

\begin{theorem}[Minkowski's second theorem]
\label{Mink2nd}
For any lattice $\Lambda \subseteq \widetilde{\mathcal{K}}^d$ and $i=1,...d$, let 
$$\lambda_i(\Lambda)=\min\{r>0: \text{ there exist }i \text{ linearly independent vectors in } \Lambda \text{ of norm }\leq r\}.$$
Then,
\begin{equation*}
    \prod_{i=1}^d \lambda_i(\Lambda)=\det(\Lambda).
\end{equation*}
\end{theorem}
Throughout this paper, $B(\mathbf{x},r)$ stands for the closed ball of radius $r>0$ centered at $\mathbf{x}\in\widetilde{\mathcal{K}}^d$, that is, $B(\mathbf{x},r)=\{\mathbf{y}\in\widetilde{\mathcal{K}}^d : \|\mathbf{x}-\mathbf{y}\|\leq r\}$. 
The following lattice point counting result will be used for the upper bound of $\dim_H\mathbf{DI}_d(\varepsilon)$.
\begin{proposition}\label{Prop_BK}\cite[Lemma 6.2]{BK}
    Let $\Lambda\subseteq \widetilde{\mathcal{K}}^d$ be a lattice. Then
    \[
    \# (\Lambda\cap B(0,r)) = \prod_{i=1}^d \left\lceil \frac{qr}{\lambda_i(\Lambda)} \right\rceil.
    \]
\end{proposition}

We shall prove positive characteristic analogues of the results of \cite[Section 2]{CC16}. Firstly, given $X\subseteq \widetilde{\mathcal{K}}^d$, we denote the radius of the largest ball disjoint from $X$ by
$$e(X)=\sup\{r>0:\exists \alpha\in \widetilde{\mathcal{K}}^d: B(\alpha,r)\cap X=\emptyset\}$$
Let $\mathcal{O}$ be the maximal compact subgroup of $\widetilde{\mathcal{K}}$, that is
$$\mathcal{O}=\mathbb{F}_q\left[\left[x^{-1}\right]\right]=\{f\in \widetilde{\mathcal{K}}:\vert f\vert\leq 1\}.$$
Then, the group of isometries for the $\operatorname{SL}_d(\widetilde{\mathcal{K}})$ action on $\mathcal{L}_d$ is $\operatorname{SL}_d(\mathcal{O})$. 
\begin{lemma}
The function $e$ is $\operatorname{SL}_d(\mathcal{O})$ invariant.
\end{lemma}
\begin{proof}
    Denote $r:=e(X)$ and let $u\in \operatorname{SL}_d(\mathcal{O})$. Assume towards a contradiction that for every $\alpha\in\widetilde{\mathcal{K}}^d$, $B(\alpha,r)\cap uX\neq \emptyset$. Let $\mathbf{v}\in B(\alpha,r)\cap uX$. Then, $u^{-1}\mathbf{v}\in X$ and 
    $$\Vert u^{-1}\alpha-u^{-1}\mathbf{v}\Vert=\Vert \alpha-\mathbf{v}\Vert\leq r.$$
    Hence, for every $\alpha\in \widetilde{\mathcal{K}}^d$, $B(u^{-1}\alpha,r)\cap X\neq \emptyset$, which is a contradiction to the definition of $r$. Thus, $e(uX)\geq e(X)$, and by replacing $u$ with $u^{-1}$ and $X$ with $uX$, we obtain that $e(X)=e(uX)$.
\end{proof}
\begin{lemma}
\label{e(X)Lattice}
For any lattice $\Lambda\subseteq \widetilde{\mathcal{K}}^d$,
\begin{equation*}
    q^2e(\Lambda)=\lambda_d(\Lambda)
\end{equation*}
\end{lemma}
\begin{proof}
Since $e$ is $\operatorname{SL}_d(\mathcal{O})$ invariant, by \cite[Theorem 1.24]{A23}, it suffices to assume that $\Lambda=\mathcal{R}x^{\log_q\lambda_1(\Lambda)}\mathbf{e}_1\oplus\dots \oplus \mathcal{R}x^{\log_q\lambda_d(\Lambda)}\mathbf{e}_d$, where $\mathbf{e}_1,\dots,\mathbf{e}_d$ is the canonical basis for $\widetilde{\mathcal{K}}^d$. Define $\mathbf{v}=\sum_{i=1}^dx^{\log_q\lambda_i(\Lambda)-1}\mathbf{e}_i$. We claim that
\begin{equation*}
    B\left(\mathbf{v},\frac{1}{q^2}\lambda_d(\Lambda)\right)\cap \Lambda=\emptyset.
\end{equation*}
Assume towards a contradiction that $\mathbf{u}=\sum_{i=1}^da_ix^{\log_q\lambda_i(\Lambda)}\mathbf{e}_i\in B\left(\mathbf{v},\frac{1}{q^2}\lambda_d(\Lambda)\right)\cap \Lambda$, where $a_1\dots a_d\in \mathcal{R}$. Then,
\begin{equation}
\label{eqn:maxCoorInBall}
    \Vert \mathbf{u}-\mathbf{v}\Vert=\max_{i=1,\dots,d}\lambda_i(\Lambda)\left|a_i-\frac{1}{x}\right|\leq q^{-2}\lambda_d(\Lambda).
\end{equation}
In particular, (\ref{eqn:maxCoorInBall}) holds for $i=d$, so that $\left|a_i-\frac{1}{x}\right|\leq q^{-2}$, but this cannot hold for $a_i\in \mathcal{R}$. Hence, $B\left(\mathbf{v},\frac{1}{q^2}\lambda_d(\Lambda)\right)\cap \Lambda \neq \emptyset$, so that $e(\Lambda)\geq q^{-2}\lambda_d(\Lambda)$. 

On the other hand, by \cite[Theorem 1.25]{A23},
$$\Lambda+B\left(0,\frac{1}{q}\lambda_d(\Lambda)\right)=\widetilde{\mathcal{K}}^d.$$
Thus, every ball of radius $\frac{1}{q}\lambda_d(\Lambda)$ intersects $\Lambda$. Hence, $e(\Lambda)\leq q^{-2}\lambda_d(\Lambda)$.
\end{proof}
\begin{lemma}
\label{e(hyp)}
    Let $L\subseteq \widetilde{\mathcal{K}}^d$, let $L'\subset L$ be a sublattice of dimension $d-1$ with minimal volume and let $H'=\operatorname{span}_{\widetilde{\mathcal{K}}}(L')$. Then, $q^2e(L+H')=\lambda_d(L)$. 
\end{lemma}
\begin{proof}
    Since $e$ is $\operatorname{SL}_d(\mathcal{O})$ invariant, then, by \cite[Theorem 1.24]{A23} , it suffices to assume that $L=\mathcal{R}x^{\log_q\lambda_1(L)}\mathbf{e}_1\oplus\dots\oplus \mathcal{R}x^{\log_q\lambda_d(L)}\mathbf{e}_d$.
    In this case, Theorem \ref{Mink2nd} implies that $$H'=\operatorname{span}_{\widetilde{\mathcal{K}}}\big\{x^{\log_q\lambda_1(L)}\mathbf{e}_1,\dots ,x^{\log_q\lambda_{d-1}(L)}\mathbf{e}_{d-1}\big\}.$$ 
    Thus, 
    \begin{equation*}
        L+H'=x^{\log_q\lambda_d(L)}\mathbf{e}_d+H'.
    \end{equation*}
    If $\mathbf{v}=\sum_{i=1}^da_ix^{\log_q\lambda_d(L)}\mathbf{e}_i\in B\left(x^{\log_q\lambda_d(L)-1},\frac{1}{q^2}\lambda_d(L)\right)\cap (L+H')$, where $a_d\in \mathcal{R}$, then, 
    \begin{equation*}
        \lambda_d(L)\left|a_d-x^{-1}\right|=\left| a_dx^{\log_q\lambda_d(L)}-x^{\log_q\lambda_d(L)-1}\right|\leq \frac{1}{q^2}\lambda_d(L).
    \end{equation*}
    Thus, $\left|a_d-x^{-1}\right|\leq q^{-2}$, but this cannot hold for $a_d\in \mathcal{R}$. Hence, $$B\left(x^{\log_q\lambda_d(L)}\mathbf{e}_d,\frac{1}{q^2}\lambda_d(L)\right)\cap (L+H')=\emptyset,$$ 
    so that $e(L+H')\geq \frac{1}{q^2}\lambda_d(L)$. On the other hand, since $L\subseteq L+H'$, then, $e(L+H')\leq e(L)=\frac{1}{q^2}\lambda_d(L)$. Hence, $e(L+H')=\frac{1}{q^2}\lambda_d(L)$. 
    \end{proof}

\section{Self-similar coverings}
\label{sec:SSCov}
In this section, we recall the definition of self-similar structures and some related dimension results following \cite{C07,CC16}.
\begin{definition}\label{Def_SelfSimilar}
Let $X$ be a metric space with a metric $d_X$. A \textit{self-similar structure} on $X$ is a triple $(J,\sigma,B)$, where $J$ is countable, $\sigma\subseteq J\times J$, and $B$ is a map from $J$ into the set of bounded subsets of $X$.
For $\alpha\in J$, we denote 
\begin{equation*}
    \sigma(\alpha)=\{\beta\in J:(\alpha,\beta)\in \sigma\}.
\end{equation*}
We say that a sequence $\{\alpha_k\}_{k\in \mathbb{N}}\subseteq J$ is \textit{$\sigma$-admissible} if for every $k\in \mathbb{N}$, $\alpha_{k+1}\in \sigma(\alpha_k)$.  For a bounded subset $E$ of $X$, we denote $\operatorname{diam}(E)=\sup\{d_X(y,z):y,z\in E\}$.

For a subset $S$ of $X$, a \textit{self-similar covering} of $S$ is a self-similar structure $(J,\sigma,B)$ such that for every $\theta\in S$, there exists a $\sigma$-admissible sequence $\{\alpha_k\}_{k\in \mathbb{N}}$ in $J$ satisfying
\begin{itemize}
    \item $\lim_{k\to \infty}\operatorname{diam}B(\alpha_k)=0$;
    \item $\bigcap_{k\in \mathbb{N}}B(\alpha_k)=\{\theta\}$.
\end{itemize}
In this case, we say that the self-similar structure $(J,\sigma,B)$ covers the subset $S$.

We call a triple $(J,\sigma,B)$ a \textit{strictly nested self-similar structure} on $X$ if it is a self-similar structure on $X$ satisfying  
\begin{itemize}
    \item for all $\alpha\in J$, $\sigma(\alpha)$ is finite, $B(\alpha)$ is a nonempty compact subset of $X$, and for all $\beta\in\sigma(\alpha)$ we have $B(\beta)\subset B(\alpha)$;
    \item for all $\sigma$-admissible sequences $(\alpha_k)_{k\in\mathbb{N}}$, $\lim_{k\to\infty}\operatorname{diam}B(\alpha_k)=0$;
    \item for all $\alpha\in J$ and each $\beta\in\sigma(\alpha)$, $\operatorname{diam}B(\beta)<\operatorname{diam}B(\alpha)$.
\end{itemize}
\end{definition}

\begin{theorem}\cite[Theorem 5.3]{C07} 
\label{dim_HUppBnd}
    Let $X$ be a metric space, $S$ be a subset of $X$, $(J,\sigma,B)$ be a self-similar covering of $S$, and $s>0$. If for every $\alpha\in J$, 
    \[
        \sum_{\beta\in \sigma(\alpha)}\operatorname{diam}(B(\beta))^s\leq \operatorname{diam}(B(\alpha))^s,
    \]
    then $\dim_H(S)\leq s$. 
\end{theorem}
\begin{theorem}
\label{dim_HlowBnd}
\cite[Theorem 3.4]{CC16}
    Let $X$ be a metric space and $S$ be a subset of $X$. Suppose that there is a strictly nested self-similar structure $(J,\sigma,B)$ that covers a subset of $S$ and constants $c,s\geq 0$ and $\rho\in (0,1)$ such that
    \begin{enumerate}
        \item\label{item_lower_1} for all $\alpha\in J$ and all $\beta\in \sigma(\alpha)$, there are at most $c$ points $\gamma$ in $\sigma(\alpha)\setminus \{\beta\}$ such that
        $$d_X\left(B(\beta),B(\gamma)\right)\leq \rho\operatorname{diam}B(\alpha);$$
        \item\label{item_lower_2} for all $\alpha\in J,$
        \[
            \sum_{\beta\in \sigma(\alpha)}\operatorname{diam}B(\beta)^s\geq (c+1)\operatorname{diam}B(\alpha)^s.
        \]
    \end{enumerate}
    Then $S$ contains a subset of positive $s$-dimensional Hausdorff measure.
\end{theorem}
\begin{theorem}\cite[Theorem 3.6]{CC16}
\label{thm:CCLow3.6}
  Let $X$ be a metric space and $S$ be a subset of $X$. Suppose that there is a strictly nested self-similar structure $(J,\sigma,B)$ that covers a subset of $S$, a subset $J_0\subset J$ that contains a tail of any $\sigma$-admissible sequence, a function $\rho:J\to (0,1)$, and two constants $c,s\geq 0$ such that
  \begin{enumerate}
      \item\label{item_wlower_1} for all $\alpha\in J_0$ and all $\beta\in\sigma(\alpha)$, there are at most $c$ points $\gamma$ in $\sigma(\alpha)\setminus\{\beta\}$ such that
      \[
    d_X(B(\beta),B(\gamma))\leq\rho(\alpha)\operatorname{diam}B(\alpha);
      \]
      \item\label{item_wlower_2} for all $\alpha\in J_0$,
      \[
      \sum_{\beta\in\sigma(\alpha)}\left(\rho(\beta)\operatorname{diam}B(\beta)\right)^s \geq (c+1)\left(\rho(\alpha)\operatorname{diam}B(\alpha)\right)^s;
      \]
      \item\label{item_wlower_3} for all $\alpha\in J_0$ and all $\beta\in\sigma(\alpha)$, $\rho(\beta)\operatorname{diam}B(\beta)<\rho(\alpha)\operatorname{diam}B(\alpha)$.
      \end{enumerate}
      Then $S$ contains a subset of positive $s$-dimensional Hausdorff measure.
\end{theorem}
\section{Best approximation vectors, orthogonality, and Farey lattices}
\label{sec:bestApprox}
In this section, we recall the definition of best approximation vectors and prove several properties of best approximations. Then, we introduce orthogonality and Farey lattices over function fields and prove some related properties. In particular, we are able to obtain stronger properties than those in the real case due to Lemma \ref{umIneq}. 

\subsection{Best approximation vectors}
\label{subsec:bestApprox}
We denote 
$$Q=\{(a_1,\dots, a_d,b)\in \mathcal{R}^{d+1}:\gcd(a_1,\dots,a_d,b)=1,b\neq 0\}.$$
For $\mathbf{u}=(a,b)=(a_1,\dots, a_d,b)\in Q$, we denote $\hat{\mathbf{u}}=\frac{a}{b}$ and $\vert \mathbf{u}\vert=\vert b\vert$. Given $\boldsymbol{\theta}\in \widetilde{\mathcal{K}}^d$ and $\mathbf{u}\in Q$, we denote $A(\boldsymbol{\theta},\mathbf{u})=\Vert b\boldsymbol{\theta}-a\Vert$. 
\begin{definition}
\label{def:BestApprox}
    We say that $\mathbf{u}=(a,b)\in Q$ is a \textit{best approximation of $\boldsymbol{\theta}\in \widetilde{\mathcal{K}}^d$} if the following conditions hold.
    \begin{enumerate}
        \item \label{bestApproxInc}For every $\mathbf{v}\in Q$ with $\vert \mathbf{v}\vert<\vert \mathbf{u}\vert$, $A(\boldsymbol{\theta},\mathbf{u})<A(\boldsymbol{\theta},\mathbf{v})$;
        \item For every $\mathbf{v}\in Q$ with $\vert \mathbf{v}\vert\leq \vert\mathbf{u}\vert$, $A(\boldsymbol{\theta},\mathbf{u})\leq A(\boldsymbol{\theta},\mathbf{v})$. 
    \end{enumerate}
    For each $\boldsymbol{\theta}\in \widetilde{\mathcal{K}}^d\setminus \mathcal{K}^d$, we associate \textit{a sequence $\{\mathbf{u}_n\}_{n\in \mathbb{N}}$ of best approximation vectors for $\boldsymbol{\theta}$} with the following properties:
    \begin{enumerate}
        \item $|\mathbf{u}_1|=1$;
        \item $|\mathbf{u}_n|<|\mathbf{u}_{n+1}|$ for all $n\in\mathbb{N}$;
        \item there is no best approximation $\mathbf{u}\in Q$ of $\boldsymbol{\theta}$ with $|\mathbf{u}_n|<|\mathbf{u}|<|\mathbf{u}_{n+1}|$ for any $n\in\mathbb{N}$.
    \end{enumerate}
\end{definition}
\begin{remark}\label{Remark_uniqueBA}
    See \cite{BZ,KLP} for the existence of a sequence of best approximation vectors for $\boldsymbol{\theta}\in \widetilde{\mathcal{K}}^d\setminus\mathcal{K}^d$.  
    We note that multiplying by $\mathbb{F}_q^*$ does not effect the norm of a vector. Thus, if $\mathbf{u}=(a,b)\in Q$ is a best approximation of $\boldsymbol{\theta}\in \widetilde{\mathcal{K}}^d$, then every element of $\mathbb{F}_q^*\mathbf{u}$ is also a best approximation. Furthermore, $b$ is unique up to $\mathbb{F}_q^*$-multiples in the following sense. If $\mathbf{v}=(a',b')\in Q$ satisfies $\vert b\vert=\vert b'\vert$ and $b'\notin \mathbb{F}_q^*b$, then $\mathbf{v}$ is not a best approximation of $\boldsymbol{\theta}$. To see this, assume that $\mathbf{v}$ is a best approximation of $\boldsymbol{\theta}$. Note that there exists some element $\alpha\in \mathbb{F}_q^*$ such that $0<|b-\alpha b'|<|b|$. Since $\mathbf{u}$ and $\alpha\mathbf{v}$ are best approximations of $\boldsymbol{\theta}$ with $|\mathbf{u}|=|\alpha\mathbf{v}|$, we have $A(\boldsymbol{\theta},\mathbf{u})=A(\boldsymbol{\theta},\alpha\mathbf{v})$, hence,
    $$\|(b-\alpha b')\boldsymbol{\theta} - (a-\alpha a')\|\leq \max\{A(\boldsymbol{\theta},\mathbf{u}),A(\boldsymbol{\theta},\alpha \mathbf{v})\}=A(\boldsymbol{\theta},\mathbf{u}).$$ 
    This is a contradiction to the definition of a best approximation.
\end{remark}
We note that by Definition \ref{def:BestApprox}(\ref{bestApproxInc}), if $\{\mathbf{u}_n\}_{n\in \mathbb{N}}$ is a sequence of best approximation vectors, then, the sequence $\{ A(\boldsymbol{\theta},\mathbf{u}_n)\}_{n\in \mathbb{N}}$ is strictly decreasing. 
\begin{lemma}
\label{A(alpha,v)<=A(hat(u),v)}
    Assume that $\mathbf{u}\in Q$ is a best approximation of $\boldsymbol{\theta}\in \widetilde{\mathcal{K}}^d$. Then, for every $\mathbf{v}\in Q$ such that $\vert \mathbf{v}\vert\leq \vert \mathbf{u}\vert$ and $\mathbf{v}\notin \mathbb{F}_q^*\mathbf{u}$, we have
    $$A(\boldsymbol{\theta},\mathbf{v})=A(\hat{\mathbf{u}},\mathbf{v}).$$
\end{lemma}
\begin{proof}
    Write $\mathbf{u}=(a,b)$ and $\mathbf{v}=(a',b')$.
    Since $\vert b'\vert\leq \vert b\vert$, by Lemma \ref{umIneq} it follows that
    \begin{equation}
    \begin{split}
    \label{eqn:A(theta,v)Bnd}
        A(\boldsymbol{\theta},\mathbf{v})&=\Vert b'\boldsymbol{\theta}-a'\Vert=\vert b'\vert\Bigg\Vert \boldsymbol{\theta}-\frac{a'}{b'}\Bigg\Vert
        \leq \vert b'\vert\cdot \max\Bigg\{\bigg\Vert \boldsymbol{\theta}-\frac{a}{b}\bigg\Vert,\bigg\Vert \frac{a'}{b'}-\frac{a}{b}\bigg\Vert\Bigg\}\\
        &\leq \max\Bigg\{\Vert b\boldsymbol{\theta}-a\Vert, \bigg\Vert b'\frac{a}{b}-a'\bigg\Vert\Bigg\}
        =\max\{A(\boldsymbol{\theta},\mathbf{u}),A(\hat{\mathbf{u}},\mathbf{v})\}.
    \end{split}
    \end{equation}
    By Remark \ref{Remark_uniqueBA}, since $b$ is unique up to $\mathbb{F}_q^*$-multiples, if $b'\notin \mathbb{F}_q^* b$, then we have
    $A(\boldsymbol{\theta},\mathbf{u})<A(\boldsymbol{\theta},\mathbf{v})$. We note that $b$ is unique as soon as the value $A(\boldsymbol{\theta},\mathbf{u})$ is fixed. Thus, the right hand side of (\ref{eqn:A(theta,v)Bnd}) is equal to $A(\hat{\mathbf{u}},\mathbf{v})$, that is, $A(\boldsymbol{\theta},\mathbf{v})\leq A(\hat{\mathbf{u}},\mathbf{v})$. Suppose that $b'\in \mathbb{F}_q^* b$ and we may assume that $b'=b$ and $A(\boldsymbol{\theta},\mathbf{u})=A(\boldsymbol{\theta},\mathbf{v})$.
    Since $\mathbf{v}\notin \mathbb{F}_q^* \mathbf{u}$ and $A(\boldsymbol{\theta},\mathbf{u})\leq 1$, it follows that $A(\hat{\mathbf{u}},\mathbf{v})=\|a-a'\|\geq 1 \geq A(\boldsymbol{\theta},\mathbf{v})$.

    On the other hand, by Lemma \ref{umIneq},
    \begin{equation*}
    \begin{split}
        A(\hat{\mathbf{u}},\mathbf{v})&=\Bigg\Vert b'\frac{a}{b}-a'\Bigg\Vert
        \leq \max\Bigg\{A(\boldsymbol{\theta},\mathbf{v}),\left|\frac{b'}{b}\right|\Vert b\boldsymbol{\theta}-a\Vert\Bigg\}\\
        &\leq \max\{A(\boldsymbol{\theta},\mathbf{v}),A(\boldsymbol{\theta},\mathbf{u})\}=A(\boldsymbol{\theta},\mathbf{v}).
    \end{split}
    \end{equation*}
    Therefore, we have $A(\hat{\mathbf{u}},\mathbf{v})=A(\boldsymbol{\theta},\mathbf{v})$. 
\end{proof}
\begin{definition}\label{Def_bestappro2}
    For $\mathbf{u}\in Q$ and $r>0$, denote 
    $$
    \mathcal{B}(\mathbf{u},r)=\{\boldsymbol{\theta}\in \widetilde{\mathcal{K}}^d:A(\boldsymbol{\theta},\mathbf{u})<r\}
    \quad\text{and}\quad\bar{\mathcal{B}}(\mathbf{u},r)=\{\boldsymbol{\theta}\in \widetilde{\mathcal{K}}^d:A(\boldsymbol{\theta},\mathbf{u})\leq r\}.
    $$
    For $\mathbf{v}\in Q$, define
    $$\Delta_{\mathbf{v}}(\mathbf{u})=\{\boldsymbol{\theta}\in \widetilde{\mathcal{K}}^d:A(\boldsymbol{\theta},\mathbf{u})<A(\boldsymbol{\theta},\mathbf{v})\}
    \quad\text{and}\quad
    \bar{\Delta}_{\mathbf{v}}(\mathbf{u})=\{\boldsymbol{\theta}\in \widetilde{\mathcal{K}}^d:A(\boldsymbol{\theta},\mathbf{u})\leq A(\boldsymbol{\theta},\mathbf{v})\}.
    $$
    Then, $\mathbf{u}$ is a best approximation vector for $\boldsymbol{\theta}$ if and only if $\boldsymbol{\theta}\in \Delta(\mathbf{u})$, where 
    $$\Delta(\mathbf{u})=\left(\bigcap_{\vert \mathbf{v}\vert<\vert \mathbf{u}\vert}\Delta_{\mathbf{v}}(\mathbf{u})\right)\cap \left(\bigcap_{\vert \mathbf{v}\vert=\vert \mathbf{u}\vert}
    \bar{\Delta}_{\mathbf{v}}(\mathbf{u})\right).$$
    In a view of Lemma \ref{A(alpha,v)<=A(hat(u),v)}, we define 
    $$r(\mathbf{u})=\min_{\substack{
    \mathbf{v}\in Q: |\mathbf{v}|\leq |\mathbf{u}|,\\
     \mathbf{v}\notin \mathbb{F}_q^*\mathbf{u}
    }}A(\hat{\mathbf{u}},\mathbf{v}).$$
\end{definition}
\begin{lemma}
\label{BestApproxIncl}
    For $\mathbf{u}\in Q$,
    we have 
    $$\mathcal{B}(\mathbf{u},r(\mathbf{u})) \subseteq
    \Delta(\mathbf{u}) \subseteq \bar{\mathcal{B}}(\mathbf{u},r(\mathbf{u})).$$
\end{lemma}
\begin{proof}
    If $\boldsymbol{\theta}\in \Delta(\mathbf{u})$, that is, $\mathbf{u}$ is a best approximation vector for $\boldsymbol{\theta}$, 
    then it follows from Lemma \ref{A(alpha,v)<=A(hat(u),v)} that for every $\mathbf{v}\in Q$ with $\vert \mathbf{v}\vert\leq \vert \mathbf{u}\vert$ and $\mathbf{v}\notin \mathbb{F}_q^*\mathbf{u}$,
    we have $$A(\boldsymbol{\theta},\mathbf{u})\leq A(\boldsymbol{\theta},\mathbf{v})= A(\hat{\mathbf{u}},\mathbf{v}).$$ 
    Hence, $\Delta(\mathbf{u})\subseteq \bar{\mathcal{B}}(\mathbf{u},r(\mathbf{u}))$.
    
    If $\boldsymbol{\theta}\in \mathcal{B}(\mathbf{u},r(\mathbf{u}))$, 
    then we have $A(\boldsymbol{\theta},\mathbf{u})<A(\hat{\mathbf{u}},\mathbf{v})$ 
    for every $\mathbf{v}\in Q$ with $\vert \mathbf{v}\vert\leq \vert \mathbf{u}\vert$ and $\mathbf{v}\notin \mathbb{F}_q^*\mathbf{u}$.
    It follows that for every such $\mathbf{v}$, 
    \begin{equation*}
        A(\boldsymbol{\theta},\mathbf{u})
        <A(\hat{\mathbf{u}},\mathbf{v})
        =\vert \mathbf{v}\vert \Vert \hat{\mathbf{u}}-\hat{\mathbf{v}}\Vert
        \leq \vert \mathbf{v}\vert\max\left\{\Vert \boldsymbol{\theta}-\hat{\mathbf{u}}\Vert,\Vert \boldsymbol{\theta}-\hat{\mathbf{v}}\Vert\right\}
        \leq \max\{A(\boldsymbol{\theta},\mathbf{u}),A(\boldsymbol{\theta},\mathbf{v})\}.
    \end{equation*}
    Hence, $A(\boldsymbol{\theta},\mathbf{u})<A(\boldsymbol{\theta},\mathbf{v})$, that is, $\boldsymbol{\theta}\in \Delta_{\mathbf{v}}(\mathbf{u})$ for every such $\mathbf{v}$.  
    Since $A(\boldsymbol{\theta},\mathbf{u})=A(\boldsymbol{\theta},\mathbf{v})$ for every $\mathbf{v}\in \mathbb{F}_q^*\mathbf{u}$, we finally have
    $\boldsymbol{\theta}\in \Delta(\mathbf{u})$. 
\end{proof}
\begin{lemma}
\label{A(u_i+j,u_i)<lambda_1(u_i+1)}
    Let $\boldsymbol{\theta}\in \widetilde{\mathcal{K}}^d\setminus \mathcal{K}^d$ and let $\{\mathbf{u}_{n}\}_{n\in \mathbb{N}}$ be a sequence of best approximation vectors for $\boldsymbol{\theta}$. Then, for every $i,j\in \mathbb{N}$, we have 
    \begin{equation*}
        A(\hat{\mathbf{u}}_{i+j},\mathbf{u}_i)= A(\boldsymbol{\theta},\mathbf{u}_i)=r(\mathbf{u}_{i+1}).
    \end{equation*}
\end{lemma}
\begin{proof}
    For the first equality, observe that by Lemma \ref{A(alpha,v)<=A(hat(u),v)}, $A(\boldsymbol{\theta},\mathbf{v})=A(\hat{\mathbf{u}}_{i+j},\mathbf{v})$ for every $\mathbf{v}\in Q$ with $\vert \mathbf{v}\vert\leq \vert \mathbf{u}_{i+j}\vert$ and $\mathbf{v}\notin \mathbb{F}_q^*\mathbf{u}_{i+j}$. In particular, this holds for $\mathbf{v}=\mathbf{u}_i$, which gives rise to the first equality. 
    
    For the second part, by Lemma \ref{A(alpha,v)<=A(hat(u),v)}, we have
    \begin{equation*}
    \label{eqn:r(u_i+1)Ineq}
    r(\mathbf{u}_{i+1})
    =\min_{\substack{
    \mathbf{v}\in Q: \vert \mathbf{v}\vert\leq \vert \mathbf{u}_{i+1}\vert,\\
    \mathbf{v}\notin \mathbb{F}_q^*\mathbf{u}_{i+1}
    }}A(\hat{\mathbf{u}}_{i+1},\mathbf{v})
    =\min_{\substack{
    \mathbf{v}\in Q:\vert \mathbf{v}\vert\leq \vert \mathbf{u}_{i+1}\vert,\\
    \mathbf{v}\notin \mathbb{F}_q^*\mathbf{u}_{i+1}\\}}A(\boldsymbol{\theta},\mathbf{v}).
    \end{equation*}
    Then it is clear that $r(\mathbf{u}_{i+1})\leq A(\boldsymbol{\theta},\mathbf{u}_i)$.
    Denoting $\mathbf{u}_{i+1}=(a,b)$ and using Remark \ref{Remark_uniqueBA},
    it follows from the definition of best approximations that $A(\boldsymbol{\theta},\mathbf{v})\geq A(\boldsymbol{\theta},\mathbf{u}_i)$ for every $\mathbf{v}=(a',b')\in Q$ with $|\mathbf{v}|\leq |\mathbf{u}_{i+1}|$ and $b'\notin \mathbb{F}_q^* b$.
    For every $\mathbf{v}=(a',b')\in Q$ with $b'\in\mathbb{F}_q^* b$ but $\mathbf{v}\notin \mathbb{F}_q^* \mathbf{u}_{i+1}$, we have $A(\hat{\mathbf{u}}_{i+1},\mathbf{v})\geq 1$.
    Hence, we have the second equality of the lemma. 
\end{proof}

\begin{corollary}
\label{Cor_SingChar}
Let $\boldsymbol{\theta}\in \widetilde{\mathcal{K}}^d\setminus \mathcal{K}^d$ and let $\{\mathbf{u}_{n}\}_{n\in \mathbb{N}}$ be the sequence of best approximation vectors for $\boldsymbol{\theta}$. Then
$\boldsymbol{\theta}\in \mathbf{DI}_d(\varepsilon)$ if and only if  $|\mathbf{u}_n|^{1/d}r(\mathbf{u}_n)<\varepsilon$ for all large enough $n\geq 1$.
In particular, $\boldsymbol{\theta}\in \mathbf{Sing}_d$ if and only if $|\mathbf{u}_n|^{1/d}r(\mathbf{u}_n)\to 0$ as $n\to \infty$.
\end{corollary}
\begin{proof}
Write $\mathbf{u}_n = (a_n,b_n)$. It is clear that 
$\boldsymbol{\theta}\in \mathbf{DI}_d(\varepsilon)$ if and only if $|b_n|^{1/d}\|b_{n-1}\boldsymbol{\theta}-a_{n-1}\|<\varepsilon$ for all large enough $n\geq 2$. Hence, the corollary follows from Lemma \ref{A(u_i+j,u_i)<lambda_1(u_i+1)}.
\end{proof}

\subsection{Orthogonality and Farey lattice}
\label{subsec:Orth}
One of the main differences from the real case is orthogonality. We use the following definition of orthogonality from \cite[Lemma 3.3]{KST}.
\begin{definition}\label{Def_Orth}
    A subset $\{\mathbf{x}_1,\dots,\mathbf{x}_k\}$ of $\widetilde{\mathcal{K}}^d$ is said to be \textit{orthogonal} if 
    \[
    \|r_1\mathbf{x}_1+\cdots+r_k\mathbf{x}_k\|=\max\{|r_1|\|\mathbf{x}_1\|,\dots,|r_k|\|\mathbf{x}_k\|\}
    \]
    for any $r_1,\dots,r_k \in \widetilde{\mathcal{K}}$.
\end{definition}
By the Hadamard inequality \cite[Lemma 2.4]{RW} and Theorem \ref{Mink2nd}, any lattice $\Lambda$ in $\widetilde{\mathcal{K}}^d$ admits an orthogonal basis $\xi_1,\dots,\xi_d$ of the lattice $\Lambda$ with $\|\xi_i\|=\lambda_i(\Lambda)$ for each $i=1,\dots,d$. This corresponds to a ``Minkowski reduced basis'' in the real case.

\begin{definition}\label{Def_Fareylattice}
    For $\mathbf{u}\in Q$, we define the \textit{Farey lattice} $\Lambda_{\mathbf{u}}$ in $\widetilde{\mathcal{K}}^d$ by 
$$\Lambda_{\mathbf{u}}=\operatorname{span}_{\mathcal{R}}\{\mathbf{e}_1,\dots, \mathbf{e}_d,\hat{\mathbf{u}}\}=\pi_{\mathbf{u}}(\mathcal{R}^{d+1}),$$
where $\pi_{\mathbf{u}}(a,b)=b\hat{\mathbf{u}}-a$, where $(a,b)\in \widetilde{\mathcal{K}}^d\times\widetilde{\mathcal{K}}$.  
Note that by Lemma \ref{Lemma_requalLam}, $\det(\Lambda_{\mathbf{u}})=\vert \mathbf{u}\vert^{-1}$. We define 
$$\hat{\lambda}_i(\Lambda_{\mathbf{u}})=\vert \mathbf{u}\vert^{\frac{1}{d}}\lambda_i(\Lambda_{\mathbf{u}}) \quad\text{for every } i=1,\dots, d,$$
and we say that $\hat{\lambda}_i$ are the \textit{normalized successive minima} of $\Lambda_{\mathbf{u}}$.

For $\mathbf{u}\in Q$, let $\Lambda_{\mathbf{u}}'$ be the $d-1$ dimensional sublattice of $\Lambda_{\mathbf{u}}$ of minimal covolume, and denote $H_{\mathbf{u}}'=\operatorname{span}_{\widetilde{\mathcal{K}}}(\Lambda_{\mathbf{u}}')$ and $H_{\mathbf{u}}=\pi_{\mathbf{u}}^{-1}(H_{\mathbf{u}}')$. It can be easily checked that $\Lambda_{\mathbf{u}}' = \mathrm{span}_\mathcal{R}\{\xi_1,\dots,\xi_{d-1}\}$ and $H_{\mathbf{u}}' = \mathrm{span}_{\widetilde{\mathcal{K}}}\{\xi_1,\dots,\xi_{d-1}\}$.
\end{definition}

One of the reasons we consider the Farey lattice is because $\lambda_1(\Lambda_{\mathbf{u}})$ is indeed equal to $r(\mathbf{u})$ in Definition \ref{Def_bestappro2}.
\begin{lemma}\label{Lemma_requalLam}
    For $\mathbf{u}\in Q$, we have $\det(\Lambda_{\mathbf{u}})=|\mathbf{u}|^{-1}$ and $\lambda_1(\Lambda_{\mathbf{u}})=r(\mathbf{u}).$
\end{lemma}
\begin{proof}
In order to prove the first equality, let $\mathbf{u}=(a,b)$ and consider the group homomorphism $\varphi:\Lambda_\mathbf{u} \to \mathcal{R}/b\mathcal{R}$ given by
\[
\varphi\left(\sum_{i=1}^d c_i \mathbf{e}_i + c_{d+1}\hat{\mathbf{u}}\right)=c_{d+1}+b\mathcal{R}.
\]
Since $\gcd(a,b)=1$, the map $\varphi$ is well-defined. Since $\varphi$ is surjective and $\mathrm{Ker}(\varphi)=\mathcal{R}^d$, we have $[\Lambda_\mathbf{u}:\mathcal{R}^d]=\#\mathcal{R}/b\mathcal{R}=|b|$, where $[\Lambda_\mathbf{u}:\mathcal{R}^d]$ denotes the index of the subgroup $\mathcal{R}^d$ in $\Lambda_\mathbf{u}$. Hence, it follows from $\det(\mathcal{R}^d)=1$ that $\det(\Lambda_\mathbf{u})=|b|^{-1}=|\mathbf{u}|^{-1}$.

Now, let us prove the second equality. First, observe that
\[
\lambda_1(\Lambda_{\mathbf{u}})=\min_{\substack{(a,b)\in \mathcal{R}^d\times \mathcal{R} \\ b\hat{\mathbf{u}}-a \neq 0}}\| b\hat{\mathbf{u}}-a\| = \min_{\mathbf{v}\in Q, \mathbf{v}\notin \mathbb{F}_q^* \mathbf{u}} A(\hat{\mathbf{u}},\mathbf{v}) \leq r(\mathbf{u}).
\]
If $\lambda_1(\Lambda_{\mathbf{u}}) = A(\hat{\mathbf{u}},\mathbf{v})$ for some $\mathbf{v} \in Q$ with $|\mathbf{v}|>|\mathbf{u}|$, then writing $\mathbf{u}=(a,b),\mathbf{v}=(a',b')$ and using Euclidean division of polynomials, we have $b'=kb+r$ for some $k,r\in\mathcal{R}$ with $0<|r|<|b|$. Since $A(\hat{\mathbf{u}},\mathbf{v}) = A(\hat{\mathbf{u}},\mathbf{v}-k\mathbf{u})$ and $0<|\mathbf{v}-k\mathbf{u}|<|\mathbf{u}|$, we have $\lambda_1(\Lambda_{\mathbf{u}}) \geq r(\mathbf{u})$.
\end{proof}

The following two lemmas are a function field analogue of the results of \cite[Section 5.2]{CC16} for the geometry associated with the subspace $H_\mathbf{u}$.
\begin{lemma}
\label{H_uCond}
    Let $\mathbf{u},\mathbf{v}\in Q$. Then, $\mathbf{v}\in H_{\mathbf{u}}$ if and only if $\hat{\mathbf{v}}\in \hat{\mathbf{u}}+H_{\mathbf{u}}'$.
\end{lemma}
\begin{proof}
    Notice that $\mathbf{v}\in H_{\mathbf{u}}$ if and only if $\pi_{\mathbf{u}}(\mathbf{v})\in H_{\mathbf{u}}'$. Hence, if we write $\mathbf{v}=(a,b)$, then it is equivalent to $b\hat{\mathbf{u}}-a\in H_{\mathbf{u}}'$. Hence, $b^{-1}\pi_{\mathbf{u}}(\mathbf{v})=\hat{\mathbf{u}}-\hat{\mathbf{v}}\in H_{\mathbf{u}}'$. Thus, it is equivalent to $\hat{\mathbf{v}}\in \hat{\mathbf{u}}+H_{\mathbf{u}}'$. 
\end{proof}
\begin{lemma}
\label{alphaLen}
    Let $\mathbf{u},\mathbf{v}\in Q$ be such that $\mathbf{v}\in H_{\mathbf{u}}$ and $\Vert \hat{\mathbf{u}}-\hat{\mathbf{v}}\Vert\leq \frac{\lambda_1(\Lambda_{\mathbf{u}})}{\vert \mathbf{u}\vert}$, and let $\alpha\in \Lambda_{\mathbf{v}}\setminus H_{\mathbf{u}}'$. Then we have $\|\alpha\| \geq \lambda_d(\Lambda_{\mathbf{u}})\geq q^{-2}\lambda_d(\Lambda_{\mathbf{v}})$.
\end{lemma}
\begin{proof}
    Since $\mathbf{v}\in H_{\mathbf{u}}$, by Lemma \ref{H_uCond}, we have $\hat{\mathbf{v}}-\hat{\mathbf{u}}\in H_{\mathbf{u}}'$. Hence, $\Lambda_{\mathbf{v}}\subseteq \Lambda_{\mathbf{u}}+\widetilde{\mathcal{K}}(\hat{\mathbf{v}}-\hat{\mathbf{u}})\subseteq \Lambda_{\mathbf{u}}+H_{\mathbf{u}}'$. 
    Let $\xi_1,\dots,\xi_d$ be an orthogonal basis of $\Lambda_\mathbf{u}$ with $\|\xi_i\|=\lambda_i(\Lambda_\mathbf{u})$ for each $i=1,\dots,d$.
    Then we have $\alpha\in (\Lambda_{\mathbf{u}}+H_{\mathbf{u}}')\setminus H_{\mathbf{u}}'=(\mathcal{R}\setminus\{0\})\xi_d+H_{\mathbf{u}}'$. Since $H_{\mathbf{u}}'=\mathrm{span}_{\widetilde{\mathcal{K}}}\{\xi_1,\dots,\xi_{d-1}\}$, it follows from the orthogonality of $\{\xi_1,\dots,\xi_d\}$ that $\Vert \alpha\Vert\geq \lambda_d(\Lambda_\mathbf{u})$. 
    
    We now claim that $e(\Lambda_{\mathbf{v}})\leq q^2 e(\Lambda_{\mathbf{u}})$. Then, by Lemma \ref{e(X)Lattice}, we have 
    $$\Vert \alpha\Vert\geq \lambda_d(\Lambda_\mathbf{u})=q^2e(\Lambda_{\mathbf{u}}) \geq e(\Lambda_{\mathbf{v}}) = q^{-2}\lambda_d(\Lambda_{\mathbf{v}}),$$
    which completes the proof.

    Observe that by definition of $e(\Lambda_{\mathbf{u}})$, for any $\beta\in \widetilde{\mathcal{K}}^d$, $B(\beta,q^2e(\Lambda_{\mathbf{u}}))\cap \Lambda_{\mathbf{u}} \neq \emptyset$. Thus, there exist $(a,b)\in \mathcal{R}^d \times \mathcal{R}$ with $|b|\leq |\mathbf{u}|$, such that $\|\beta - (b\hat{\mathbf{u}}+a)\|\leq q^2e(\Lambda_{\mathbf{u}})$.
    Since $\Vert \hat{\mathbf{u}}-\hat{\mathbf{v}}\Vert\leq \frac{\lambda_1(\Lambda_{\mathbf{u}})}{\vert \mathbf{u}\vert}$ and $q^2 e(\Lambda_{\mathbf{u}})=\lambda_d(\Lambda_{\mathbf{u}})\geq \lambda_1(\Lambda_{\mathbf{u}})$, we have
    $$\|\beta - (b\hat{\mathbf{v}}+a)\|\leq \max\{\|\beta - (b\hat{\mathbf{u}}+a)\|, \|b(\hat{\mathbf{u}}-\hat{\mathbf{v}})\|\}\leq \max\{q^2e(\Lambda_{\mathbf{u}}), \lambda_1(\Lambda_{\mathbf{u}})\}=q^2e(\Lambda_{\mathbf{u}}).$$
    Hence, $e(\Lambda_{\mathbf{v}})\leq q^2e(\Lambda_{\mathbf{u}})$, which completes the proof. 
\end{proof}

We will often consider inverse images of the map $\pi_\mathbf{u}$ of lattice vectors in $\Lambda_{\mathbf{u}}$. Therefore, we will need the following lemma.
\begin{lemma}
\label{lem:pi^-1(Val)}
    Let $\mathbf{u}\in Q$, $\alpha\in \Lambda_{\mathbf{u}}$, and $k\in\mathbb{Z}_{\geq 0}$. Then we have 
    $$\#\{\mathbf{v}\in \pi^{-1}_{\mathbf{u}}(\alpha)\cap Q:\vert \mathbf{v}\vert=q^k\vert \mathbf{u}\vert\}\leq (q-1)q^{k}.$$
Moreover, if $\alpha\in \Lambda_{\mathbf{u}}$ is primitive, then
 $$\#\{\mathbf{v}\in \pi^{-1}_{\mathbf{u}}(\alpha)\cap Q:\vert \mathbf{v}\vert=q^k\vert \mathbf{u}\vert\}= (q-1)q^{k}.$$
\end{lemma}
\begin{proof}
Denote $\mathbf{u}=(a_0,b_0)$. Using Euclidean division of polynomials, we can choose $(a_1,b_1)\in \pi_{\mathbf{u}}^{-1}(\alpha)\cap \mathcal{R}^{d+1}$ with $|b_1|<|b_0|$.
We first claim that
    $$\{\mathbf{v}\in \pi_{\mathbf{u}}^{-1}(\alpha)\cap \mathcal{R}^{d+1}:\vert \mathbf{v}\vert=q^k\vert \mathbf{u}\vert\}=\big\{(a_1+a_0s,b_1+b_0s): s\in \mathcal{R}\text{ with }\vert s\vert=q^k\big\}.$$
    It is clear that the right hand side is contained in the left hand side. For $\mathbf{v}=(a,b)\in \mathcal{R}^{d+1}$ with $\pi_{\mathbf{u}}(\mathbf{v})=\alpha$ and $|\mathbf{v}|=q^k|\mathbf{u}|$, we have $\alpha=b\frac{a_0}{b_0}-a=b_1\frac{a_0}{b_0}-a_1$.
Thus, $(b-b_1)\frac{a_0}{b_0}=a-a_1$.
Since $\gcd(a_0,b_0)=1$, it follows that $b_0$ divides $b-b_1$. Hence, there exists some $s\in \mathcal{R}$ such that $b=b_1+b_0s$. Thus, $a=a_1+a_0s$. Since $\vert b_1\vert<\vert b_0\vert$, we have 
$$q^k\vert \mathbf{u}\vert=\vert \mathbf{v}\vert=\vert b\vert=\vert b-b_1\vert=\vert b_0s\vert=\vert s\vert \vert \mathbf{u}\vert,$$
which completes the proof of the claim. 

Since $\#\{s\in\mathcal{R}:|s|=q^k\}=(q-1)q^{k}$, the first argument of the lemma follows from this claim. For the second argument, it is enough to show that if $\alpha\in\Lambda_{\mathbf{u}}$ is primitive, then $(a_1+a_0s,b_1+b_0s)\in Q$ for any $s\in\mathcal{R}$ with $\vert s\vert=q^k$. Indeed, if not, $(a_1+a_0s,b_1+b_0s)=s_0\mathbf{w}$ for some $s_0\in \mathcal{R}$ with $|s_0|>1$ and $\mathbf{w}\in Q$. Since $\pi_{\mathbf{u}}$ is linear, we have
\[
\alpha=\pi_{\mathbf{u}}(a_1+a_0s,b_1+b_0s)=s_0 \pi_{\mathbf{u}}(\mathbf{w}).
\]
This is a contradiction to the primitivity of $\alpha$. Hence, the second argument follows.

\end{proof}

\section{Upper Bound}
\label{sec:UpBnd}
In order to estimate the upper bound of the Hausdorff dimension of $\mathbf{DI}_d(\varepsilon)$, we will define the following self-similar structure for $\mathbf{DI}_d(\varepsilon)$.
\begin{definition} 
    For $\mathbf{u}\in Q$, we define
$$D(\mathbf{u})=\bigg\{\mathbf{v}\in Q:\ \vert \mathbf{u}\vert\leq \vert \mathbf{v}\vert,\ \mathbf{v}\in H_{\mathbf{u}},\ \text{and } \Vert \hat{\mathbf{u}}-\hat{\mathbf{v}}\Vert\leq \frac{\lambda_1(\Lambda_{\mathbf{u}})}{\vert \mathbf{u}\vert}\bigg\}.$$
For $\mathbf{u}\in Q$ and $\mathbf{v}\in D(\mathbf{u})$, we define
\[
E(\mathbf{u},\mathbf{v},\varepsilon)=\Bigg\{\mathbf{w}\in Q:\ \vert \mathbf{v}\vert<\vert \mathbf{w}\vert,\ \mathbf{w}\notin H_{\mathbf{u}},\ \text{and } \Vert \hat{\mathbf{v}}-\hat{\mathbf{w}}\Vert<\frac{\varepsilon}{\vert \mathbf{v}\vert \vert \mathbf{w}\vert^{1/d}}\Bigg\},
\] and 
\[\sigma_{\varepsilon}(\mathbf{u})=\bigcup_{\mathbf{v}\in D(\mathbf{u})}E(\mathbf{u},\mathbf{v},\varepsilon),\quad
Q_{\varepsilon}=\{\mathbf{u}\in Q:\hat{\lambda}_1(\Lambda_{\mathbf{u}})<\varepsilon\},
\quad B(\mathbf{u})=B\left(\hat{\mathbf{u}},\frac{1}{\vert \mathbf{u}\vert^{1+1/d}}\right).
\]
\end{definition}
\begin{remark}\label{Rem_selfsimilar}
We will consider the triple $(Q_{\varepsilon},\sigma_{\varepsilon},B)$ as a self-similar structure on $\widetilde{\mathcal{K}}^d$. Recall Definition \ref{Def_SelfSimilar} and note that $(\mathbf{u},\mathbf{w})\in \sigma_\varepsilon$ if and only if $\mathbf{w}\in E(\mathbf{u},\mathbf{v},\varepsilon)$ for some $\mathbf{v}\in D(\mathbf{u})$. Observe that $\sigma_{\varepsilon}(\mathbf{u})\subseteq Q_{\varepsilon}$ for every $\mathbf{u}\in Q$ since for any $\mathbf{w}\in\sigma_\varepsilon(\mathbf{u})$ we have
$$\hat{\lambda}_1(\mathbf{w})\leq \vert \mathbf{w}\vert^{\frac{1}{d}}\Vert\pi_{\mathbf{w}}(\mathbf{v})\Vert=\vert \mathbf{w}\vert^{\frac{1}{d}}\vert \mathbf{v}\vert\cdot \Vert \hat{\mathbf{v}}-\hat{\mathbf{w}}\Vert<\varepsilon.$$
Thus, the triple $(Q_{\varepsilon},\sigma_{\varepsilon},B)$ is indeed a self-similar structure.
\end{remark}
For given $\varepsilon>0$, it follows from Corollary \ref{Cor_SingChar} and Lemma \ref{Lemma_requalLam} that 
$$\mathbf{DI}_d(\varepsilon)\setminus \mathcal{K}^d = \{\boldsymbol{\theta}\in \widetilde{\mathcal{K}}^d \setminus\mathcal{K}^d:  \mathbf{u}_n \in Q_{\varepsilon} \text{ for all large enough } n\geq 1 \},$$
where $\{\mathbf{u}_n\}_{n\geq 1}$ is the sequence of best approximation vectors for $\boldsymbol{\theta}$. We denote by $\widetilde{\mathbf{DI}}_d(\varepsilon)$ the above set in the equation. Since the Hausdorff dimension of the rationals $\mathcal{K}^d$ is zero, we can work with $\widetilde{\mathbf{DI}}_d(\varepsilon)$ instead of $\mathbf{DI}_d(\varepsilon)$. We also define
$$\widetilde{\mathbf{DI}}_d^*(\varepsilon)=\{\boldsymbol{\theta}=(\theta_1,\dots,\theta_d)\in \widetilde{\mathbf{DI}}_d(\varepsilon):1,\theta_1,\dots, \theta_d\text{ are linearly independent over } \mathcal{K}\}.$$


\begin{proposition}
\label{ssCover}
The triple $(Q_{\varepsilon},\sigma_{\varepsilon},B)$ is a self-similar covering of $\widetilde{\mathbf{DI}}_d^*(\varepsilon)$. 
\end{proposition}
\begin{proof}
    Fix $\boldsymbol{\theta}\in \widetilde{\mathbf{DI}}_d^*(\varepsilon)$. We now claim that the sequence of best approximations $\{\mathbf{u}_n\}_{n\in \mathbb{N}}$ of $\boldsymbol{\theta}$ contains a $\sigma_{\varepsilon}$-admissible sequence. By definition, there exists some $n_0$ such that $\mathbf{u}_n\in Q_{\varepsilon}$ for every $n> n_0$. Choose $n_1=n_0+1$ and assume that we have already constructed a $\sigma_{\varepsilon}$-admissible sequence $\mathbf{u}_{n_1},\dots, \mathbf{u}_{n_i}$. We let $n_{i+1}$ be the smallest $m>n_i$ such that $\mathbf{u}_m\notin H_{\mathbf{u}_{n_i}}$. Since $1,\theta_1\dots \theta_d$ are linearly independent over $\mathcal{K}$, such $n_{i+1}$ indeed exists. Now we let $\mathbf{u}=\mathbf{u}_{n_i}$, $\mathbf{v}=\mathbf{u}_{n_{i+1}-1}$, and $\mathbf{w}=\mathbf{u}_{n_{i+1}}$. Then, $\vert \mathbf{u}\vert\leq \vert \mathbf{v}\vert<\vert \mathbf{w}\vert$, $\mathbf{v}\in H_{\mathbf{u}}$, and $\mathbf{w}\notin H_{\mathbf{u}}$. Moreover, it follows from Lemmas \ref{A(u_i+j,u_i)<lambda_1(u_i+1)} and \ref{Lemma_requalLam} that
    $$\Vert \hat{\mathbf{u}}-\hat{\mathbf{v}}\Vert=\frac{A(\hat{\mathbf{v}},\mathbf{u})}{|\mathbf{u}|}<\frac{\lambda_1(\Lambda_{\mathbf{u}})}{\vert \mathbf{u}\vert}$$
    and 
    $$\Vert \hat{\mathbf{v}}-\hat{\mathbf{w}}\Vert
    =\frac{A(\hat{\mathbf{w}},\mathbf{v})}{|\mathbf{v}|}
    =\frac{\lambda_1(\Lambda_{\mathbf{w}})}{\vert \mathbf{v}\vert}=\frac{\hat{\lambda}_1(\Lambda_{\mathbf{w}})}{\vert \mathbf{v}\vert\cdot \vert \mathbf{w}\vert^{1/d}}<\frac{\varepsilon}{\vert \mathbf{v}\vert\cdot \vert \mathbf{w}\vert^{1/d}}.$$
    Thus, $\mathbf{u}_{n_{i+1}}\in \sigma_{\varepsilon}(\mathbf{u}_{n_i})$ so that $\{\mathbf{u}_{n_i}\}_{i\in \mathbb{N}}$ is $\sigma_{\varepsilon}$-admissible. 

    It is obvious that $\mathrm{diam}\ B(\mathbf{u}_{n_i})\to 0$ as $i\to\infty$. We shall now prove that $\bigcap_{i\in\mathbb{N}}B(\mathbf{u}_{n_i}) = \{\boldsymbol{\theta}\}$.
    Note that Theorem \ref{Mink2nd} implies $\lambda_1(\Lambda_{\mathbf{u}})\leq |\mathbf{u}|^{-1/d}$ for any $\mathbf{u}\in Q$.
    By Lemmas \ref{BestApproxIncl} and \ref{Lemma_requalLam}, for any best approximation vector $\mathbf{u}$ of $\boldsymbol{\theta}$, we have
    $$\Delta(\mathbf{u})\subseteq \bar{\mathcal{B}}(\mathbf{u},r(\mathbf{u})) = B\left(\hat{\mathbf{u}},\frac{\lambda_1(\Lambda_{\mathbf{u}})}{\vert \mathbf{u}\vert}\right)\subseteq B\left(\hat{\mathbf{u}},\vert\mathbf{u}\vert^{-(1+1/d)}\right)=B(\mathbf{u}).$$
    This implies that $\boldsymbol{\theta}\in B(\mathbf{u})$ for every best approximation vector $\mathbf{u}$ of $\boldsymbol{\theta}$. Therefore, we have $\bigcap_{i\in\mathbb{N}}B(\mathbf{u}_{n_i}) = \{\boldsymbol{\theta}\}$.
\end{proof}

\begin{proposition}
\label{D(u)Cnt}
    Let $t>d$ and $\mathbf{u}\in Q$. Then, 
    $$\sum_{\mathbf{v}\in D(\mathbf{u})}\left(\frac{\vert \mathbf{u}\vert}{\vert \mathbf{v}\vert}\right)^t\leq\frac{(q-1)q^{t-1}}{q^{t-d}-1}.$$
\end{proposition}
\begin{proof}
    For $k\geq 0$, denote
    $$D_k(\mathbf{u})=\{\mathbf{v}\in D(\mathbf{u}):\vert \mathbf{v}\vert=q^k\vert \mathbf{u}\vert\}.$$
    Then, 
    \begin{equation}
    \label{eqn:D_ksum}
        \sum_{\mathbf{v}\in D(\mathbf{u})}\left(\frac{\vert \mathbf{u}\vert}{\vert \mathbf{v}\vert}\right)^t=\sum_{k=0}^{\infty}\sum_{\mathbf{v}\in D_k(\mathbf{u})}\frac{1}{q^{kt}}=\sum_{k=0}^{\infty}\frac{\#D_k(\mathbf{u})}{q^{kt}}.
    \end{equation}
    To estimate $\# D_k(\mathbf{u})$, 
    consider the map $\pi_{\mathbf{u}}$ restricted to $D_k(\mathbf{u})$. By definition, we have 
    \[
    \pi_{\mathbf{u}}(D_k(\mathbf{u})) \subseteq \Lambda_{\mathbf{u}}'=\Lambda_{\mathbf{u}}\cap \pi_{\mathbf{u}}(H_{\mathbf{u}})\quad\text{and}\quad
    \|\pi_{\mathbf{u}}(\mathbf{v})\|=\vert \mathbf{v}\vert\Vert \hat{\mathbf{u}}-\hat{\mathbf{v}}\Vert \leq q^k\lambda_1(\Lambda_{\mathbf{u}})
    \] for any $\mathbf{v}\in D_k(\mathbf{u})$.
    Hence, it follows that $\pi_{\mathbf{u}}(D_k(\mathbf{u}))\subseteq B(0,q^k\lambda_1(\Lambda_{\mathbf{u}}))\cap \Lambda_{\mathbf{u}}'$. 
    
    Let $\xi_1,\dots,\xi_d$ be an orthogonal basis of the lattice $\Lambda_{\mathbf{u}}$ with $\|\xi_i\|=\lambda_i(\Lambda_\mathbf{u})$ for each $i=1,\dots,d$ as in the paragraph after Definition \ref{Def_Orth}. Since $\Lambda_{\mathbf{u}}' = \mathrm{span}_\mathcal{R}\{\xi_1,\dots,\xi_{d-1}\}$, it follows that
    \[\begin{split}
    \#\left(B(0,q^k\lambda_1(\Lambda_{\mathbf{u}}))\cap \Lambda_{\mathbf{u}}'\right)&=\#\left\{\sum_{i=1}^{d-1}a_i\xi_i:a_i\in \mathcal{R},\vert a_i\vert\leq q^k\frac{\lambda_1(\Lambda_{\mathbf{u}})}{\lambda_i(\Lambda_{\mathbf{u}})}\right\}\\
    &\leq \prod_{i=1}^{d-1}\#\{a_i\in\mathcal{R}:|a_i|\leq q^k\} = q^{(k+1)(d-1)}. 
    \end{split}\]
    Thus, it follows from Lemma \ref{lem:pi^-1(Val)} that
    \begin{equation}
    \label{eqn:D_kEst}
        \#D_k(\mathbf{u})\leq (q-1)q^{k}\#\left(B(0,q^k\lambda_1(\Lambda_{\mathbf{u}})\cap \Lambda_{\mathbf{u}}')\right)
        \leq (q-1)q^{d-1}q^{kd}.
    \end{equation}
    Hence, by plugging in (\ref{eqn:D_kEst}) into (\ref{eqn:D_ksum}), we obtain that
    \begin{equation*}
        \sum_{\mathbf{v}\in D(\mathbf{u})}\left(\frac{\vert \mathbf{u}\vert}{\vert \mathbf{v}\vert}\right)^t\leq (q-1)q^{d-1}\sum_{k=0}^{\infty}q^{-k(t-d)}=\frac{(q-1)q^{d-1}}{1-q^{-(t-d)}} =  \frac{(q-1)q^{t-1}}{q^{t-d}-1}.
    \end{equation*}
\end{proof}
\begin{proposition}
    \label{E(u,v,epsilon)Cnt}
    For every $t>d$, we have
    $$\sum_{\mathbf{w}\in E(\mathbf{u},\mathbf{v},\varepsilon)}\left(\frac{\vert \mathbf{v}\vert}{\vert \mathbf{w}\vert}\right)^t\leq \frac{(q-1)(q+q^2)^d \varepsilon^d}{q^{t-d}-1}.$$
\end{proposition}
\begin{proof}
    For $k\geq 1$, denote
    $$E_k(\mathbf{u},\mathbf{v},\varepsilon)=\{\mathbf{w}\in E(\mathbf{u},\mathbf{v},\varepsilon):\vert \mathbf{w}\vert=q^k\vert \mathbf{v}\vert\}.$$
    Then, 
    \begin{equation}
    \label{eqn:E_ksum}
        \sum_{\mathbf{w}\in E(\mathbf{u},\mathbf{v},\varepsilon)}\left(\frac{\vert \mathbf{v}\vert}{\vert \mathbf{w}\vert}\right)^t=\sum_{k=1}^{\infty}\sum_{\mathbf{w}\in E_k(\mathbf{u},\mathbf{v},\varepsilon)}\frac{1}{q^{kt}}=\sum_{k=1}^{\infty}\frac{\#E_k(\mathbf{u},\mathbf{v},\varepsilon)}{q^{kt}}.
    \end{equation}
    Observe that for any $\mathbf{w}\in E_k(\mathbf{u},\mathbf{v},\varepsilon)$, 
    $$\|\pi_{\mathbf{v}}(\mathbf{w})\|=\vert \mathbf{w}\vert\Vert \hat{\mathbf{v}}-\hat{\mathbf{w}}\Vert < |\mathbf{w}|\frac{\varepsilon}{|\mathbf{v}||\mathbf{w}|^{1/d}}=\varepsilon q^{k\frac{d-1}{d}}|\mathbf{v}|^{-1/d}.$$
    Since $\mathbf{v}\in H_{\mathbf{u}}$ and $\mathbf{w}\notin H_{\mathbf{u}}$, it follows from Lemma \ref{H_uCond} that $\pi_{\mathbf{v}}(\mathbf{w}) \notin H_{\mathbf{u}}'$.
    Hence, we have 
    $$\pi_{\mathbf{v}}\left(E_k(\mathbf{u},\mathbf{v},\varepsilon)\right)\subseteq B\left(0,\varepsilon q^{k\frac{d-1}{d}}|\mathbf{v}|^{-1/d}\right)\cap \Lambda_{\mathbf{v}}\setminus H_\mathbf{u}'.$$ 
    If the set $B\left(0,\varepsilon q^{k\frac{d-1}{d}}|\mathbf{v}|^{-1/d}\right)\cap \Lambda_{\mathbf{v}}\setminus H_\mathbf{u}'$
    is non-empty, then it follows from Lemma \ref{alphaLen} that 
    \begin{equation}\label{eqn:lambdaBound}
    \lambda_d(\Lambda_{\mathbf{v}}) \leq q^2 \varepsilon q^{k\frac{d-1}{d}}|\mathbf{v}|^{-1/d}.
    \end{equation}
    It follows from Theorem \ref{Mink2nd}, Proposition \ref{Prop_BK}, Lemma \ref{lem:pi^-1(Val)}, and \eqref{eqn:lambdaBound} that 
    \begin{equation}
    \begin{split}
    \label{eqn:E_kEst}
        \#E_k(\mathbf{u},\mathbf{v},\varepsilon)
        &\leq (q-1)q^{k}\#\left(B(0,\varepsilon q^{ k\frac{d-1}{d}}\vert \mathbf{v}\vert^{-\frac{1}{d}})\cap\Lambda_{\mathbf{v}}\setminus H_{\mathbf{u}}'\right)\\
        &\leq(q-1)q^{k}\prod_{i=1}^{d}\Bigg\lceil\frac{q \varepsilon q^{k\frac{d-1}{d}}|\mathbf{v}|^{-1/d}}{\lambda_i(\Lambda_{\mathbf{v}})}\Bigg\rceil\\
        &\leq (q-1)q^{k}\prod_{i=1}^{d}\left(\frac{q \varepsilon q^{k\frac{d-1}{d}}|\mathbf{v}|^{-1/d}+\lambda_d(\Lambda_{\mathbf{v}})}{\lambda_i(\Lambda_{\mathbf{v}})}\right)\\
        &\leq (q-1)q^{k}\prod_{i=1}^{d}\frac{(q+q^2) \varepsilon q^{k\frac{d-1}{d}}|\mathbf{v}|^{-1/d}}{\lambda_i(\Lambda_{\mathbf{v}})}\\
        &=(q-1)q^{k}(q+q^2)^d \varepsilon^d q^{k(d-1)}. 
    \end{split}
    \end{equation}
    Hence, by plugging in (\ref{eqn:E_kEst}) into (\ref{eqn:E_ksum}), we obtain that
    \begin{equation*}
    \begin{split}
        \sum_{\mathbf{w}\in E(\mathbf{u},\mathbf{v},\varepsilon)}\left(\frac{\vert \mathbf{v}\vert}{\vert \mathbf{w}\vert}\right)^t
        \leq \sum_{k=1}^{\infty}\frac{(q-1)(q+q^2)^d \varepsilon^d q^{kd}}{q^{kt}}
        &=(q-1)(q+q^2)^d \varepsilon^d \sum_{k=1}^{\infty}q^{-k(t-d)}\\
        &=\frac{(q-1)(q+q^2)^d \varepsilon^d}{q^{t-d}-1}.
    \end{split}
    \end{equation*}
\end{proof}
\begin{corollary}
\label{lem:DIUpp}
    For every $\varepsilon>0$, we have 
    $$\dim_H \mathbf{DI}_d(\varepsilon) \leq \frac{d^2}{d+1}+\frac{d}{d+1}\log_q\left(1+\sqrt{(q-1)^2(q+1)^dq^{2d} \varepsilon^d}\right).
    $$
    In particular, we have $\dim_H \mathbf{Sing}_d\leq \frac{d^2}{d+1}$.
\end{corollary}
\begin{proof}
    For any $\frac{d^2}{d+1}<s\leq d$, it follows from Propositions \ref{D(u)Cnt} and \ref{E(u,v,epsilon)Cnt} with $t=\frac{d+1}{d}s$ that
    \begin{equation}
    \label{eqn:HdimUppBnd}
    \begin{split}
        \sum_{\mathbf{w}\in \sigma_{\varepsilon}(\mathbf{u})}\frac{\operatorname{diam}(B(\mathbf{w}))^s}{\operatorname{diam}(B(\mathbf{u}))^s}
        &=\sum_{\mathbf{w}\in \sigma_{\varepsilon}(\mathbf{u})}\left(\frac{\vert \mathbf{u}\vert}{\vert \mathbf{w}\vert}\right)^{(1+1/d)s}\\
        &\leq \sum_{\mathbf{v}\in D(\mathbf{u})}\left(\frac{\vert \mathbf{u}\vert}{\vert \mathbf{v}\vert}\right)^{(1+1/d)s}\sum_{\mathbf{w}\in E(\mathbf{u},\mathbf{v},\varepsilon)}\left(\frac{\vert \mathbf{v}\vert}{\vert \mathbf{w}\vert}\right)^{(1+1/d)s}\\
        &\leq \frac{(q-1)^2(q+q^2)^d\varepsilon^{d}q^{\frac{d+1}{d}s-1}}{\left(q^{\frac{d+1}{d}s-d}-1\right)^2}
        \leq \frac{(q-1)^2(q+1)^dq^{2d}\varepsilon^{d}}{\left(q^{\frac{d+1}{d}s-d}-1\right)^2}.
    \end{split}
    \end{equation}
    If $s\geq \frac{d^2}{d+1}+\frac{d}{d+1}\log_q\left(1+\sqrt{(q-1)^2(q+1)^dq^{2d} \varepsilon^d}\right)$, 
    then the right hand side of (\ref{eqn:HdimUppBnd}) is $\leq 1$. Thus, by Theorem \ref{dim_HUppBnd} and Proposition \ref{ssCover}, 
    \begin{equation}\label{Eq_dim_DIT}
        \dim_H \widetilde{\mathbf{DI}}_d^*(\varepsilon) \leq \frac{d^2}{d+1}+\frac{d}{d+1}\log_q\left(1+\sqrt{(q-1)^2(q+1)^dq^{2d} \varepsilon^d}\right).
    \end{equation}
    Since the set of $\boldsymbol{\theta}\in \widetilde{\mathcal{K}}^d$ such that $1,\theta_1,\cdots,\theta_d$ are linearly dependent over $\mathcal{K}$ has Hausdorff dimension $d-1$, the corollary follows from \eqref{Eq_dim_DIT}.
\end{proof}

\section{Lower bound}
\label{sec:LowBnd}
In this section, we will estimate the lower bound of the Hausdorff dimension of $\mathbf{DI}_d(\varepsilon)$.
For simplicity, we will use the following notation throughout this section: for $\mathbf{u}\in Q$ and $i=1,\dots,d$,
\[
\lambda_i(\mathbf{u})=\lambda_i(\Lambda_{\mathbf{u}})\qquad\text{and}\qquad \hat{\lambda}_i(\mathbf{u})=\hat{\lambda}_i(\Lambda_{\mathbf{u}}).
\]

\subsection{Self-similar structure}
\label{subsec:LowSelfSim}
While a certain orthogonal projection was used in \cite{CC16}, we will utilize the orthogonal basis given in the paragraph after Definition \ref{Def_Orth}. The basic idea is to directly use the lattice given in \cite[Lemma 8.7]{CC16}. For this, we need the following definition.

Given $\mathbf{u}\in Q$, we fix an orthogonal basis $\xi_1,\dots,\xi_d$ of the lattice $\Lambda_{\mathbf{u}}$ with $\|\xi_i\|=\lambda_i(\mathbf{u})$ for each $i=1,\dots,d$ as in the paragraph after Definition \ref{Def_Orth}.
Note that by Theorem \ref{Mink2nd}, $\det(\Lambda_{\mathbf{u}})=\lambda_1(\mathbf{u})\cdots\lambda_d(\mathbf{u})=1/|\mathbf{u}|$.
\begin{definition} 
Consider $\alpha\in \Lambda_{\mathbf{u}}$ of the form $\alpha=\sum_{i<d}m_i\xi_i + n\xi_d$ with $n\neq 0$.
\begin{enumerate}
    \item Define the lattice
    \[
    \Lambda_{\alpha} = \mathcal{R}\xi_1 +\cdots + \mathcal{R}\xi_{d-1} + \mathcal{R}\sum_{i<d}\frac{m_i}{n}\xi_i.
    \]
    Recall that $H_{\mathbf{u}}' = \mathrm{span}_{\widetilde{\mathcal{K}}}\{\xi_1,\dots,\xi_{d-1}\}$.
    Define the linear map $\pi_{\alpha}:\widetilde{\mathcal{K}}^d \to H_{\mathbf{u}}'$ by
    $$\pi_{\alpha}(\xi) = c_1\xi_1+\cdots+c_{d-1}\xi_{d-1} - c_d \sum_{i<d}\frac{m_i}{n}\xi_i,$$
    where $\xi = \sum_{i\leq d}c_i\xi_i$. One can easily check that $\pi_{\alpha}(\Lambda_{\mathbf{u}})= \Lambda_{\alpha}$ and $\mathrm{Ker}(\pi_{\alpha})=\widetilde{\mathcal{K}}\alpha$.

    \item Let $\mathbf{v}\in Q$ be such that $\pi_{\mathbf{u}}(\mathbf{v})=\alpha$, and assume that $\alpha$ is primitive and $\|\alpha\|=|n|\lambda_d(\mathbf{u})$. By Lemma \ref{Lem_covol_al}, we have
\[
\det(\Lambda_{\alpha}) = \frac{1}{|\mathbf{u}|\|\alpha\|}= \frac{1}{|\mathbf{u}\wedge\mathbf{v}|},
\]
where $|\mathbf{u}\wedge\mathbf{v}|=|\mathbf{u}||\mathbf{v}|\|\hat{\mathbf{u}}-\hat{\mathbf{v}}\|$. It is worth noting that the quantity $|\mathbf{u}\wedge\mathbf{v}|$ is the supremum norm of the natural projection of $\mathbf{u}\wedge\mathbf{v}\in \bigwedge^2 \widetilde{\mathcal{K}}^{d+1}$ onto the subspace spanned by $\mathbf{e}_1\wedge \mathbf{e}_{d+1},\dots,\mathbf{e}_d\wedge \mathbf{e}_{d+1}$. For each $i=1\dots,d-1$, denote $\lambda_i(\alpha)=\lambda_i(\Lambda_{\alpha})$ and
define $\hat{\lambda}_i(\alpha)=\hat{\lambda}_i(\Lambda_{\alpha})=\vert \mathbf{u}\wedge \mathbf{v}\vert^{\frac{1}{d-1}}\lambda_i(\Lambda_{\alpha})$.
\end{enumerate}\end{definition}

We first calculate the covolume of the lattice $\Lambda_\alpha$ when $\alpha$ is primitive.
\begin{lemma}\label{Lem_covol_al}
    Let $\mathbf{u}\in Q$ and let $\alpha = \sum_{i<d}m_i\xi_i + n\xi_d \in \Lambda_{\mathbf{u}}$ be primitive with $n\neq 0$. Then $$\det(\Lambda_{\alpha})=\frac{\lambda_1(\mathbf{u})\cdots\lambda_{d-1}(\mathbf{u})}{|n|}=\frac{1}{|\mathbf{u}||n|\lambda_d(\mathbf{u})}.$$
\end{lemma}
\begin{proof}
    Note that $\Lambda_{\mathbf{u}}' = \mathrm{span}_{\mathcal{R}}\{\xi_1,\dots,\xi_{d-1}\}$ and by the orthogonality of $\{\xi_1,\dots,\xi_{d-1}\}$ we have $\det(\Lambda_\mathbf{u}')=\|\xi_1\|\cdots\|\xi_{d-1}\|=\lambda_1(\mathbf{u})\cdots\lambda_{d-1}(\mathbf{u})$.
    Since $\Lambda_{\mathbf{u}}'$ is a sublattice of $\Lambda_\alpha$ of full rank,
    \[
    \det(\Lambda_{\mathbf{u}}')=[\Lambda_\alpha:\Lambda_{\mathbf{u}}']\det(\Lambda_{\alpha}),
    \]
    where $[\Lambda_\alpha:\Lambda_{\mathbf{u}}']$ denotes the index of the subgroup $\Lambda_{\mathbf{u}}'$ of $\Lambda_\alpha$.
    Consider the group homomorphsim $\varphi:\Lambda_\alpha \to \mathcal{R}/n\mathcal{R}$ given by $$\varphi\left(\sum_{i<d}c_i\xi_i + c_d \sum_{i<d}\frac{m_i}{n}\xi_i\right)=c_d+n\mathcal{R}.$$
    Since $\alpha$ is primitive, we have $\gcd(m_1,\dots,m_{d-1},n)=1$. Hence, $\varphi$ is well-defined. 
    Since $\varphi$ is surjective and $\mathrm{Ker}(\varphi)=\Lambda_{\mathrm{u}}'$, we have $[\Lambda_\alpha:\Lambda_{\mathbf{u}}']=\#\mathcal{R}/n\mathcal{R}=|n|$. 
\end{proof}

\begin{lemma}
\label{lem:lambda_1(w))}
    Let $\mathbf{u}\in Q$, $\alpha = \sum_{i<d}m_i\xi_i + n\xi_d \in \Lambda_{\mathbf{u}}$ be primitive with $n\neq 0$, and $\mathbf{v}\in Q$ be such that $\pi_{\mathbf{u}}(\mathbf{v})=\alpha$. Assume that $\|\alpha\|=|n|\lambda_d(\mathbf{u})$. Then $\frac{|\mathbf{u}\wedge\mathbf{v}|}{|\mathbf{v}|}\leq \lambda_1(\alpha)$ implies $\lambda_1(\mathbf{v})=\frac{|\mathbf{u}\wedge\mathbf{v}|}{|\mathbf{v}|}$.
\end{lemma}
\begin{proof}
    We first claim that for any $\xi \in \widetilde{\mathcal{K}}^d$, $\|\xi\| \geq \|\pi_{\alpha}(\xi)\|$. Indeed, writing $\xi = \sum_{i\leq d}c_i\xi_i$, we have
    \[
    \|\pi_{\alpha}(\xi)\| \leq \max_{i=1,\dots,d-1}\left|c_i-c_d\frac{m_i}{n}\right|\Vert \xi_i\|
    \leq \max_{i=1,\dots,d-1} \max\Big\{|c_i|\lVert\xi_i\|,\left|c_d\frac{m_i}{n}\right|\|\xi_i\|\Big\}.
    \]
    Since $\|\alpha\|=|n|\lambda_d(\mathbf{u})$, we have $|\frac{m_i}{n}|\|\xi_i\|\leq \|\xi_d\|$. Hence by the orthogonality of $\{\xi_1,\dots,\xi_d\}$, 
    $\|\pi_{\alpha}(\xi)\| \leq \max_{i=1,\dots,d}|c_i|\|\xi_i\| = \|\xi\|$.

    Observe that since $\mathrm{Ker}(\pi_\alpha)=\widetilde{\mathcal{K}}\alpha=\widetilde{\mathcal{K}}(\widehat{\mathbf{u}}-\widehat{\mathbf{v}})$, it follows that $\pi_{\alpha}(b\widehat{\mathbf{u}}-a)=\pi_{\alpha}(b\widehat{\mathbf{v}}-a)$ for any $(a,b)\in \mathcal{R}^d\times\mathcal{R}$. Hence $\pi_{\alpha}(\Lambda_{\mathbf{u}})=\pi_{\alpha}(\Lambda_{\mathbf{v}})=\Lambda_{\alpha}$.
    
    Now, let $\beta=\pi_{\mathbf{v}}(\mathbf{u})$. For any $\gamma \in \Lambda_{\mathbf{v}}\setminus \widetilde{\mathcal{K}}\beta$, since $\mathrm{Ker}(\pi_\alpha)=\widetilde{\mathcal{K}}\beta$, it follows from the claim and the observation that $\|\gamma\| \geq \|\pi_{\alpha}(\gamma)\| \geq \lambda_1(\alpha)\geq \frac{|\mathbf{u}\wedge\mathbf{v}|}{|\mathbf{v}|}$.
 Finally, we claim that $\beta$ is primitive. Assuming this claim, if $\gamma \in \Lambda_{\mathbf{v}}\cap\widetilde{\mathcal{K}}\beta \setminus \{0\}$, then $\|\gamma\|\geq \|\beta\|=\frac{|\mathbf{u}\wedge\mathbf{v}|}{|\mathbf{v}|}$.
    Therefore, $\lambda_1(\mathbf{v})=\frac{|\mathbf{u}\wedge\mathbf{v}|}{|\mathbf{v}|}$. 
    
     In order to prove the last claim, suppose that $\beta$ is not primitive and write $\beta=f\gamma$ with $\gamma \in \Lambda_{\mathbf{v}}$ and $f\in\mathcal{R}$ with $\deg f \geq 1$. Choose $\mathbf{w}\in Q$ such that $\gamma=\pi_{\mathbf{v}}(\mathbf{w})$. By linearity of $\pi_\mathbf{v}$, $\pi_\mathbf{v}(f\mathbf{w}-\mathbf{u})=0$, hence $f\mathbf{w}-\mathbf{u}=f'\mathbf{v}$ for some $f'\in\mathcal{R}$. This implies that $f'\alpha= f\pi_{\mathbf{u}}(\mathbf{w})$. Since $\alpha$ is primitive, we have $f|f'$. This contradicts the fact that $f\mathbf{w}-f'\mathbf{v}=\mathbf{u}\in Q$.
\end{proof}

We shall now define the fractal structure for the lower bound of the Hausdorff dimension of $\mathbf{DI}_d(\varepsilon)$. 
\begin{definition}\label{Def_lowerstructure}
Given $\mathbf{u}\in Q$ and $\varepsilon>0$, define
\[
\Lambda_{\mathbf{u}}(\varepsilon)=\{\alpha\in \Lambda_{\mathbf{u}}: \alpha\text{ is primitive and } \hat{\lambda}_1(\alpha)>\varepsilon\}.
\]
Given $n\in \mathcal{R}$ and $N\in \mathbb{N}$, define
\[
C_n'({\mathbf{u}})=\Bigg\{\mathbf{v}=\sum_{i=1}^{d-1}\alpha_i\xi_i+n\xi_d:\alpha_i\in \widetilde{\mathcal{K}},\vert \alpha_i\vert<\frac{\vert n\vert\lambda_d(\mathbf{u})}{\lambda_i(\mathbf{u})}\Bigg\};\quad
C_N(\mathbf{u})=\bigcup_{0<\vert n\vert\leq q^N}C_n'(\mathbf{u}),
\]
and define

$$F_N(\mathbf{u},\varepsilon)=\bigcup_{\alpha\in \Lambda_{\mathbf{u}}(\varepsilon)\cap C_N(\mathbf{u})}\zeta(\mathbf{u},\alpha,\varepsilon);\qquad F(\mathbf{u},\varepsilon)=\bigcup_{N\in\mathbb{N}}F_N(\mathbf{u},\varepsilon),$$
where 
$$\zeta(\mathbf{u},\alpha,\varepsilon)=\Bigg\{\mathbf{v}\in Q:\pi_{\mathbf{u}}(\mathbf{v})=\alpha,\left(\frac{\vert \mathbf{u}\wedge \mathbf{v}\vert}{\varepsilon}\right)^{\frac{d}{d-1}}\leq \vert \mathbf{v}\vert\leq \left(\frac{q\vert \mathbf{u}\wedge \mathbf{v}\vert}{\varepsilon}\right)^{\frac{d}{d-1}}\Bigg\}.$$
Finally, we define 
\[
\sigma_{\varepsilon,N}(\mathbf{u})=F_N(\mathbf{u},\varepsilon),\quad
Q_{\varepsilon,N}=\bigcup_{\mathbf{u}\in Q}\sigma_{\varepsilon,N}(\mathbf{u}),\quad 
B(\mathbf{u})=B\left(\hat{\mathbf{u}},\frac{\lambda_1(\mathbf{u})}{q\vert \mathbf{u}\vert}\right). 
\]
\end{definition}

\begin{lemma}
\label{v<w}
    Let $0<\varepsilon<1$ and $\mathbf{v}\in F(\mathbf{u},\varepsilon)$. Then $\lambda_1(\mathbf{v})=\frac{|\mathbf{u}\wedge\mathbf{v}|}{|\mathbf{v}|}$, $\frac{1}{q}\varepsilon\leq \hat{\lambda}_1(\mathbf{v})\leq \varepsilon$, and $\vert \mathbf{v}\vert>\vert \mathbf{u}\vert$. 
\end{lemma}
\begin{proof}
    Let $\alpha=\pi_{\mathbf{u}}(\mathbf{v})$. Since 
    $$\frac{\varepsilon}{q\vert \mathbf{u}\wedge \mathbf{v}\vert}\leq \frac{1}{\vert \mathbf{v}\vert^{\frac{d-1}{d}}}\leq \frac{\varepsilon}{\vert \mathbf{u}\wedge \mathbf{v}\vert}$$
    and $\hat{\lambda}_1(\alpha)=\vert \mathbf{u}\wedge \mathbf{v}\vert^{\frac{1}{d-1}}\lambda_1(\alpha)>\varepsilon$, it follows from $\varepsilon<1$ that
    \begin{equation*}
        \frac{\vert \mathbf{u}\wedge \mathbf{v}\vert}{\vert \mathbf{v}\vert}\leq \frac{\varepsilon^{\frac{d}{d-1}}}{\vert \mathbf{u}\wedge \mathbf{v}\vert^{\frac{1}{d-1}}}<\varepsilon^{\frac{1}{d-1}}\lambda_1(\alpha)<\lambda_1(\alpha).
    \end{equation*}
    Since $\alpha$ is primitive, by Lemma \ref{lem:lambda_1(w))}, we have $\lambda_1(\mathbf{v})=\frac{\vert \mathbf{u}\wedge \mathbf{v}\vert}{\vert \mathbf{v}\vert}$. Thus, 
    \begin{equation*}
        \frac{\varepsilon}{q}\leq \hat{\lambda}_1(\mathbf{v})=\frac{\vert \mathbf{u}\wedge \mathbf{v}\vert}{\vert \mathbf{v}\vert^{\frac{d-1}{d}}}\leq \varepsilon.
    \end{equation*}
    Since $\alpha\in \Lambda_{\mathbf{u}}\setminus H_{\mathbf{u}}'$, by Lemma \ref{alphaLen}, we have $\Vert \alpha \Vert\geq \lambda_d(\mathbf{u})$. Note that by Theorem \ref{Mink2nd}, $\hat{\lambda}_1(\mathbf{u})\cdots\hat{\lambda}_d(\mathbf{u})=1$. Hence, $\hat{\lambda}_d(\mathbf{u})\geq 1$. Therefore, 
    \begin{equation}
    \label{eqn:hat(lambda_d)(u)Ineq}
        1\leq \hat{\lambda}_d(\mathbf{u})=\vert \mathbf{u}\vert^{1/d}\lambda_d(\mathbf{u})\leq \vert \mathbf{u}\vert^{1/d}\Vert \alpha\Vert.
    \end{equation}
    It follows from \eqref{eqn:hat(lambda_d)(u)Ineq} that
    \[
        \frac{\vert \mathbf{u}\vert}{\vert\mathbf{v}\vert}\leq \frac{\varepsilon^{\frac{d}{d-1}}\vert \mathbf{u}\vert}{\vert \mathbf{u}\wedge \mathbf{v}\vert^{\frac{d}{d-1}}}=\frac{\varepsilon^{\frac{d}{d-1}}\vert \mathbf{u}\vert}{\Vert \alpha\Vert^{\frac{d}{d-1}}\vert \mathbf{u}\vert^{\frac{d}{d-1}}}
        =\frac{\varepsilon^{\frac{d}{d-1}}}{(\vert \mathbf{u}\vert^{\frac{1}{d}}\Vert \alpha\Vert)^{\frac{d}{d-1}}}\leq  \varepsilon^{\frac{d}{d-1}}.
    \]
    Since $\varepsilon<1$, we have $\vert \mathbf{v}\vert>\vert \mathbf{u}\vert$.
\end{proof}
\begin{lemma}
\label{BallCont}
   Let $0<\varepsilon<1$ and $\mathbf{v}\in F(\mathbf{u},\varepsilon)$.  Then, $B\left(\hat{\mathbf{v}},\frac{\lambda_1(\mathbf{v})}{\vert \mathbf{v}\vert}\right)\subset B\left(\hat{\mathbf{u}},\frac{\lambda_1(\mathbf{u})}{q\vert \mathbf{u}\vert}\right)$.
\end{lemma}
\begin{proof}
We first bound $\|\hat{\mathbf{u}}-\hat{\mathbf{v}}\|$ and $\frac{\lambda_1(\mathbf{v})}{\vert \mathbf{v}\vert}$. Let $\alpha=\pi_\mathbf{u}(\mathbf{v})$. Then it follows from $\xi_1\in \Lambda_{\alpha}$ that $\lambda_1(\alpha)\leq \Vert\xi_1\Vert=\lambda_1(\mathbf{u})$. Since $\varepsilon<1$, we have
    \begin{equation}
        \label{eqn:d(v,w)BND}
        \|\hat{\mathbf{u}}-\hat{\mathbf{v}}\|
        =\frac{\vert \mathbf{u}\wedge \mathbf{v}\vert}{\vert \mathbf{u}\vert\vert \mathbf{v}\vert}
        \leq\frac{\varepsilon^{\frac{d}{d-1}}}{\vert \mathbf{u}\wedge \mathbf{v}\vert^{\frac{1}{d-1}}\vert \mathbf{u}\vert}
        < \frac{\varepsilon^{\frac{1}{d-1}} \hat{\lambda}_1(\alpha)}{\vert \mathbf{u}\wedge \mathbf{v}\vert^{\frac{1}{d-1}}\vert \mathbf{u}\vert}
        \leq \varepsilon^{\frac{1}{d-1}}\frac{\lambda_1(\mathbf{u})}{\vert \mathbf{u}\vert}< \frac{\lambda_1(\mathbf{u})}{\vert \mathbf{u}\vert}.
    \end{equation}
By Lemma \ref{v<w} and (\ref{eqn:d(v,w)BND}), we have
    \begin{equation}
    \label{eqn:lambda1(w)/wBND}
        \frac{\lambda_1(\mathbf{v})}{\vert \mathbf{v}\vert}=\frac{\vert \mathbf{u}\wedge \mathbf{v}\vert}{\vert \mathbf{v}\vert^2}=\frac{\vert \mathbf{u}\vert}{\vert \mathbf{v}\vert}\|\hat{\mathbf{u}}-\hat{\mathbf{v}}\|<\frac{\lambda_1(\mathbf{u})}{\vert \mathbf{v}\vert}<\frac{\lambda_1(\mathbf{u})}{\vert \mathbf{u}\vert}.
    \end{equation}

    Now let $\boldsymbol{\theta}\in B\left(\hat{\mathbf{v}},\frac{\lambda_1(\mathbf{v})}{\vert \mathbf{v}\vert}\right)$. By Lemma \ref{umIneq}, \eqref{eqn:d(v,w)BND}, and \eqref{eqn:lambda1(w)/wBND},
    \begin{equation*}
        \|\boldsymbol{\theta}-\hat{\mathbf{u}}\|\leq \max\{\|\hat{\mathbf{u}}-\hat{\mathbf{v}}\|,\|\boldsymbol{\theta}-\hat{\mathbf{v}}\|\}< \frac{\lambda_1(\mathbf{u})}{|\mathbf{u}|}.
    \end{equation*}
   Since the image of the norm $\|\cdot\|$ is in $q^{\mathbb{Z}}\cup \{0\}$,then, $\frac{\lambda_1(\mathbf{u})}{\vert \mathbf{u}\vert}\in q^{\mathbb{Z}}\cup\{0\}$. Hence, $\|\boldsymbol{\theta}-\hat{\mathbf{u}}\|\leq \frac{\lambda_1(\mathbf{u})}{q|\mathbf{u}|}$, that is,
   $\boldsymbol{\theta}\in B\left(\hat{\mathbf{u}},\frac{\lambda_1(\mathbf{u})}{q\vert \mathbf{u}\vert}\right)$.
\end{proof}
\begin{proposition}
\label{F(u,u)FracStruc}
    For $0<\varepsilon<1$, the triple $(Q_{\varepsilon,N},\sigma_{\varepsilon,N},B)$ is a strictly nested self-similar structure covering a subset of $\mathbf{DI}_d(\varepsilon)$.
\end{proposition}
\begin{proof}
    Note that by the definition of $Q_{\varepsilon,N}$, if $\mathbf{u}\in Q_{\varepsilon,N}$, then, $\sigma_{\varepsilon,N}(\mathbf{u})\subseteq Q_{\varepsilon,N}$. We need to verify that for every $\sigma_{\varepsilon,N}$ admissible sequence $\{\mathbf{u}_k\}_{k\in \mathbb{N}}$, 
    \begin{enumerate}
        \item \label{nested} for all $k$, $B(\mathbf{u}_{k+1})\subset B(\mathbf{u}_k)$ and $\operatorname{diam}(B(\mathbf{u}_{k+1}))<\operatorname{diam}(B(\mathbf{u}_k))$;
        \item $\lim_{k\rightarrow \infty}\operatorname{diam}(B(\mathbf{u}_k))=0$ and $\bigcap_{k=1}^{\infty}B(\mathbf{u}_k)$ is a point contained in $\mathbf{DI}_d(\varepsilon)$.
    \end{enumerate}
    By Lemma \ref{BallCont}, it is clear that (\ref{nested}) holds. Moreover, due to the definition of $\sigma_{\varepsilon,N}$ and Lemma \ref{v<w}, $\vert \mathbf{u}_k\vert\rightarrow \infty$. Therefore, $\lim_{k\rightarrow \infty}\operatorname{diam}(B(\mathbf{u}_k))=0$. Hence, $\bigcap_{k=1}^{\infty}B(\mathbf{u}_k)$ is a single point $\boldsymbol{\theta}$. 
    
    In order to prove that $\boldsymbol{\theta}\in \mathbf{DI}_d(\varepsilon)$, we observe that $\boldsymbol{\theta}\in B(\mathbf{u}_k)$ for every $k\geq 1$. Thus, for every $k\geq 1$, we have $\|\boldsymbol{\theta}-\hat{\mathbf{u}}_k\|\leq \frac{\lambda_1(\mathbf{u}_k)}{q\vert \mathbf{u}_k\vert}$. Thus, it follows that
    $$A(\boldsymbol{\theta},\mathbf{u}_k)=\|\boldsymbol{\theta}-\hat{\mathbf{u}}_k\|\vert \mathbf{u}_k\vert< \lambda_1(\mathbf{u}_k).$$
    By Lemma \ref{BestApproxIncl}, $\boldsymbol{\theta}\in \Delta(\mathbf{u}_k)$, so that for every $k\geq 1$, $\mathbf{u}_k$ is a best approximation of $\boldsymbol{\theta}$. 
    
    
    Fix any large enough $T \geq 1$ so that there is $k\geq 1$ such that $|\mathbf{u}_k| \leq T < |\mathbf{u}_{k+1}|$. 
    Since $\mathbf{u}_{k}$'s are best approximations of $\boldsymbol{\theta}$, it follows from Lemma \ref{A(u_i+j,u_i)<lambda_1(u_i+1)} and $\vert \mathbf{u}_{k+1}\vert\geq \left(\frac{\vert \mathbf{u}_k\wedge \mathbf{u}_{k+1}\vert}{\varepsilon}\right)^{\frac{d}{d-1}}$ that
    \begin{equation*}
A(\boldsymbol{\theta},\mathbf{u}_k)=A(\hat{\mathbf{u}}_{k+1},\mathbf{u}_k)=\frac{\vert \mathbf{u}_k\wedge \mathbf{u}_{k+1}\vert}{\vert \mathbf{u}_{k+1}\vert}\leq \varepsilon \vert \mathbf{u}_{k+1}\vert^{-\frac{1}{d}}< \varepsilon T^{-\frac{1}{d}}.
\end{equation*}
This proves that $\boldsymbol{\theta}\in \mathbf{DI}_d(\varepsilon)$.
\end{proof}

\subsection{Well-separatedness}
We will need the following claim regarding the distance between two balls in $\widetilde{\mathcal{K}}^d$. 
\begin{lemma}
    \label{DistBalls}
    Let $\mathbf{u},\mathbf{u}'\in \widetilde{\mathcal{K}}^d$ and let $r,r'>0$. Then,
    $$d\left(B(\mathbf{u},r),B(\mathbf{u}',r')\right)=\begin{cases}
        0&\Vert \mathbf{u}-\mathbf{u}'\Vert\leq \max\{r,r'\}\\
        \Vert \mathbf{u}-\mathbf{u}'\Vert&\text{else}.
    \end{cases}$$
\end{lemma}
\begin{proof}
    Notice that an element in $B(\mathbf{u},r)$ will be of the form $\mathbf{u}+\mathbf{v}$ where $\Vert \mathbf{v}\Vert\leq r$. Hence, if $\Vert \mathbf{u}-\mathbf{u}'\Vert\leq \max\{r,r'\}$, then, $-(\mathbf{u}-\mathbf{u}')\in B(0,\max\{r,r'\})$. Thus, 
    $$d\left(B(\mathbf{u},r),B(\mathbf{u}',r')\right)=\inf_{\Vert \mathbf{v}\Vert\leq r,\Vert \mathbf{v}'\Vert\leq r'}\Vert \mathbf{u}-\mathbf{u}'+\mathbf{v}-\mathbf{v}'\Vert\leq \Vert \mathbf{u}-\mathbf{u}'-(\mathbf{u}-\mathbf{u}')\Vert=0.$$
    On the other hand, assume that $\Vert \mathbf{u}-\mathbf{u}'\Vert>\max\{r,r'\}$. Then, for every $\Vert \mathbf{v}\Vert\leq r$ and $\Vert \mathbf{v}'\Vert\leq r'$, we have that 
    $$\Vert \mathbf{v}-\mathbf{v}'\Vert\leq \max\{\Vert \mathbf{v}\Vert,\Vert \mathbf{v}'\Vert\}\leq \max\{r,r'\}<\Vert \mathbf{u}-\mathbf{u}'\Vert.$$
    Thus, by Lemma \ref{umIneq},
    $$d\left(B(\mathbf{u},r),B(\mathbf{u}',r')\right)=\inf_{\Vert \mathbf{v}\Vert\leq r,\Vert \mathbf{v}'\Vert\leq r'}\Vert \mathbf{u}-\mathbf{u}'+\mathbf{v}-\mathbf{v}'\Vert=\Vert \mathbf{u}-\mathbf{u}'\Vert.$$
\end{proof}
\begin{lemma}\label{Lem_dist_v,w}
    Let $\mathbf{v},\mathbf{w}\in F(\mathbf{u},\varepsilon)$ be distinct. Then, $d(B(\mathbf{v}),B(\mathbf{w}))=\Vert \hat{\mathbf{v}}-\hat{\mathbf{w}}\Vert$.
\end{lemma}
\begin{proof}
    By Lemma \ref{DistBalls}, it suffices to prove that $\Vert \hat{\mathbf{v}}-\hat{\mathbf{w}}\Vert>\max\Bigg\{\frac{\lambda_1(\mathbf{w})}{q\vert \mathbf{w}\vert},\frac{\lambda_1(\mathbf{v})}{q\vert \mathbf{v}\vert}\Bigg\}$. Assume towards a contradiction that $\Vert \hat{\mathbf{w}}-\hat{\mathbf{v}}\Vert\leq \max\Bigg\{\frac{\lambda_1(\mathbf{w})}{q\vert \mathbf{w}\vert},\frac{\lambda_1(\mathbf{v})}{q\vert \mathbf{v}\vert}\Bigg\}$, and assume without loss of generality that $\frac{\lambda_1(\mathbf{v})}{\vert \mathbf{v}\vert}\leq \frac{\lambda_1(\mathbf{w})}{\vert \mathbf{w}\vert}$. It follows that
    \[
    \frac{\lambda_1(\mathbf{v})}{\vert \mathbf{w}\vert}
        \leq \frac{\Vert \pi_{\mathbf{v}}(\mathbf{w})\Vert}{\vert \mathbf{w}\vert}
        =\|\hat{\mathbf{w}}-\hat{\mathbf{v}}\|\leq \frac{\lambda_1(\mathbf{w})}{q\vert \mathbf{w}\vert}.
        \]
    Thus, $\lambda_1(\mathbf{v})\leq \frac{\lambda_1(\mathbf{w})}{q}<\lambda_1(\mathbf{w})$. 
    It follows from Lemma \ref{v<w} that $\lambda_1(\mathbf{v})=\frac{\vert \mathbf{u}\wedge \mathbf{v}\vert}{\vert \mathbf{v}\vert}=|\mathbf{u}|\|\hat{\mathbf{u}}-\hat{\mathbf{v}}\|$, $\lambda_1(\mathbf{w})=\frac{\vert \mathbf{u}\wedge \mathbf{w}\vert}{\vert \mathbf{w}\vert}=|\mathbf{u}|\|\hat{\mathbf{u}}-\hat{\mathbf{w}}\|$, and 
    \begin{equation*}
    \|\hat{\mathbf{u}}-\hat{\mathbf{w}}\|=\frac{\lambda_1(\mathbf{w})}{\vert \mathbf{u}\vert}>\frac{\lambda_1(\mathbf{w})}{\vert \mathbf{w}\vert}>\|\hat{\mathbf{w}}-\hat{\mathbf{v}}\|.
    \end{equation*}
    By Lemma \ref{umIneq} we have 
    \begin{equation*}
        \|\hat{\mathbf{w}}-\hat{\mathbf{v}}\|<\|\hat{\mathbf{u}}-\hat{\mathbf{w}}\|\leq \max\{\|\hat{\mathbf{w}}-\hat{\mathbf{v}}\|,\|\hat{\mathbf{u}}-\hat{\mathbf{v}}\|\}.
    \end{equation*}
    It follows that $\|\hat{\mathbf{w}}-\hat{\mathbf{v}}\|<\|\hat{\mathbf{u}}-\hat{\mathbf{v}}\|$, so that
    by Lemma \ref{umIneq}, $\|\hat{\mathbf{u}}-\hat{\mathbf{w}}\|=\|\hat{\mathbf{u}}-\hat{\mathbf{v}}\|$. Thus, we have
    \begin{equation*}
        \frac{\lambda_1(\mathbf{w})}{\vert \mathbf{u}\vert}=\|\hat{\mathbf{u}}-\hat{\mathbf{w}}\|=\|\hat{\mathbf{u}}-\hat{\mathbf{v}}\|=\frac{\lambda_1(\mathbf{v})}{\vert \mathbf{u}\vert}.
    \end{equation*}
    This implies that $\lambda_1(\mathbf{w})=\lambda_1(\mathbf{v})$, which is a contradiction. Hence $d\left(B(\mathbf{w}),B(\mathbf{v})\right)=\Vert \hat{\mathbf{w}}-\hat{\mathbf{v}}\Vert$.
\end{proof}
\begin{lemma}
\label{lem:B(v),B(w)Dist}
    Let $\mathbf{v},\mathbf{w}\in F_N(\mathbf{u},\varepsilon)$ be distinct. Then, 
    \begin{equation*}
        d\left(B(\mathbf{v}),B(\mathbf{w})\right)
        \geq \left(\frac{\varepsilon}{q^{N+1}\hat{\lambda}_d(\mathbf{u})}\right)^{\frac{2d}{d-1}}\frac{\lambda_1(\mathbf{u})}{\vert \mathbf{u}\vert}.
    \end{equation*}
\end{lemma}
\begin{proof}
    By Lemma \ref{v<w}, it follows that
    \begin{equation}
    \label{eqn:d(w,z)}
        \|\hat{\mathbf{v}}-\hat{\mathbf{w}}\|
        =\frac{\|\pi_{\mathbf{v}}(\mathbf{w})\|}{|\mathbf{w}|}
        \geq \frac{\lambda_1(\mathbf{v})}{|\mathbf{w}|}
        = \frac{|\mathbf{u}\wedge\mathbf{v}|}{|\mathbf{v}||\mathbf{w}|}
        = \frac{\|\pi_{\mathbf{u}}(\mathbf{v})\||\mathbf{u}|}{|\mathbf{v}||\mathbf{w}|}
        \geq \frac{\lambda_1(\mathbf{u})}{|\mathbf{u}|}\left(\frac{\vert \mathbf{u}\vert}{\vert \mathbf{v}\vert}\right)\left(\frac{\vert \mathbf{u}\vert}{\vert \mathbf{w}\vert}\right).
    \end{equation}
    By definition of $F_N(\mathbf{u},\varepsilon)$, we have
    $\vert \mathbf{v}\vert\leq  \left(\frac{q\vert \mathbf{u}\wedge \mathbf{v}\vert}{\varepsilon}\right)^{\frac{d}{d-1}}$ and $\|\pi_{\mathbf{u}}(\mathbf{v})\|\leq q^N\lambda_d(\mathbf{u})$. Thus,
    \begin{equation}\label{Eq_bound_v}
        \frac{\vert \mathbf{u}\vert}{\vert \mathbf{v}\vert}
        \geq \frac{\vert \mathbf{u}\vert \varepsilon^{\frac{d}{d-1}}}{(q\vert \mathbf{u}\wedge \mathbf{v}\vert)^{\frac{d}{d-1}}}
        =\frac{\varepsilon^{\frac{d}{d-1}}}{(q|\mathbf{u}|^{\frac{1}{d}}\|\pi_{\mathbf{u}}(\mathbf{v})\|)^{\frac{d}{d-1}}}
        \geq \left(\frac{\varepsilon}{q^{N+1}\hat{\lambda}_d(\mathbf{u})}\right)^{\frac{d}{d-1}}.
    \end{equation}
    By the same argument with $\mathbf{w}$ instead of $\mathbf{v}$, we have \eqref{Eq_bound_v} for $\mathbf{w}$.
    Combining with Lemma \ref{Lem_dist_v,w}, \eqref{eqn:d(w,z)}, and \eqref{Eq_bound_v} for $\mathbf{v}$ and $\mathbf{w}$, it follows that
    \begin{equation*}
        d\left(B(\mathbf{v}),B(\mathbf{w})\right)
        =\|\hat{\mathbf{v}}-\hat{\mathbf{w}}\|
        \geq \left(\frac{\varepsilon}{q^{N+1}\hat{\lambda}_d(\mathbf{u})}\right)^{\frac{2d}{d-1}}\frac{\lambda_1(\mathbf{u})}{\vert \mathbf{u}\vert}.
    \end{equation*}
\end{proof}

\subsection{Lower bound calculations}
\label{subsec:LowCalc}
In this subsection, we will estimate lower bounds of Hausdorff dimensions for the sets $\mathbf{DI}_d(\varepsilon)$ and $\mathbf{Sing}_d$ assuming the following proposition.  
\begin{proposition}
\label{prop:F_NSum}
    Let $0<\varepsilon<1$, $N\in \mathbb{N}$, and $\mathbf{u}\in Q_{\varepsilon,N}$. For any $s>0$, we have 
    \[
       \sum_{\mathbf{v}\in F_N(\mathbf{u},\varepsilon)} \left(\frac{\lambda_1(\mathbf{v})}{\vert \mathbf{v}\vert}\right)^s \left( \frac{\vert \mathbf{u}\vert}{\lambda_1(\mathbf{u})}\right)^s
        \geq \left(\frac{q-1}{q}-\varepsilon^{d-1}\right)\frac{\varepsilon^d(q-1)^2}{q^{d+s}}\sum_{T_0\leq k\leq T_N}q^{-(\frac{d+1}{d}s-d)k},
    \]
    where 
\begin{equation*}\label{Eq_def_T}
    T_i=\frac{d}{d-1}\left(i+1+\log_q \frac{\hat{\lambda}_d(\mathbf{u})}{\varepsilon}\right).
\end{equation*}
\end{proposition}
This proposition will be proved in \cref{Subsec_proof}.
\begin{remark}\label{rem_lambda_bound}
Note that $\hat{\lambda}_d(\mathbf{u})\geq 1$ for any $\mathbf{u}\in Q$. On the other hand, if $\mathbf{u}\in Q_{\varepsilon,N}$, then it follows from the definition of $Q_{\varepsilon,N}$ and Lemma \ref{v<w} that $\hat{\lambda}_1(\mathbf{u})\geq \frac{1}{q}\varepsilon$. Hence, by Minkowski's second theorem (Theorem \ref{Mink2nd}), we have $\hat{\lambda}_d(\mathbf{u})\leq \hat{\lambda}_1^{-(d-1)}(\mathbf{u})\leq \frac{q^{d-1}}{\varepsilon^{d-1}}$. Thus, 
\[
T_0\leq \frac{d}{d-1}\left(d+\log_q\frac{1}{\varepsilon^d}\right)=\log_q \left(\frac{q}{\varepsilon}\right)^{\frac{d^2}{d-1}}\quad\text{and}\quad T_N\geq \frac{d}{d-1}\left(N+1+\log_q\frac{1}{\varepsilon}\right).
\]
\end{remark}
\begin{corollary}\label{Cor_lowerbound}
    Let $d\geq 2$. For every $0<\varepsilon<\left(\frac{q-1}{q}\right)^{\frac{1}{d-1}}$, we have $$\dim_H\mathbf{DI}_d(\varepsilon)\geq \frac{d^2}{d+1}+\frac{d}{d+1}\log_q \left(1+\frac{(q-1)^2}{q^{2d+\frac{d^2}{d-1}}}\left(\frac{q-1}{q}-\varepsilon^{d-1}\right)\varepsilon^{d}\right).$$
\end{corollary}
\begin{proof}
    Let $(Q_{\varepsilon,N},\sigma_{\varepsilon,N},B)$ be the triple given in Definition \ref{Def_lowerstructure}, which is a strictly nested self-similar structure covering a subset of $\mathbf{DI}_d(\varepsilon)$ by Proposition \ref{F(u,u)FracStruc}. In order to use Theorem \ref{dim_HlowBnd}, we first verify the condition \eqref{item_lower_1} of Theorem \ref{dim_HlowBnd}.
    Indeed, recall that $\hat{\lambda}_d(\mathbf{u})\leq \frac{q^{d-1}}{\varepsilon^{d-1}}$ for any $\mathbf{u}\in Q_{\varepsilon,N}$ by Remark \ref{rem_lambda_bound}, hence it follows from Lemma \ref{lem:B(v),B(w)Dist} that \eqref{item_lower_1} of Theorem \ref{dim_HlowBnd} holds with $c=0$ and $\rho=q\left(\frac{\varepsilon^{d}}{q^{N+d}}\right)^{\frac{2d}{d-1}}$.
    Fix small $\delta>0$ and let $$f(\varepsilon)=\log_q\left(1+\frac{(1-\delta)(q-1)^2}{q^{2d+\frac{d^2}{d-1}}}\left(\frac{q-1}{q}-\varepsilon^{d-1}\right)\varepsilon^d\right).$$
    We claim that $\dim_H\mathbf{DI}_d(\varepsilon)\geq \frac{d^2}{d+1}+\frac{d}{d+1}f(\varepsilon)$. Since $\delta>0$ is arbitrary, this claim implies the corollary.
    By Proposition \ref{prop:F_NSum} and Remark \ref{rem_lambda_bound} with $s=\frac{d^2}{d+1}+\frac{d}{d+1}f(\varepsilon)$, it follows that for large enough $N$ and $\mathbf{u}\in Q_{\varepsilon,N}$,
    \[\begin{split}
    \sum_{\mathbf{v}\in \sigma_{\varepsilon,N}(\mathbf{u})}&\frac{\operatorname{diam}(B(\mathbf{v}))^s}{\operatorname{diam}(B(\mathbf{u}))^s}
    =\sum_{\mathbf{v}\in F_N(\mathbf{u},\varepsilon)} \left(\frac{\lambda_1(\mathbf{v})}{\vert \mathbf{v}\vert}\right)^s \left( \frac{\vert \mathbf{u}\vert}{\lambda_1(\mathbf{u})}\right)^s\\
    &\geq \left(\frac{q-1}{q}-\varepsilon^{d-1}\right)\frac{\varepsilon^d(q-1)^2}{q^{d+s}}\sum_{T_0\leq k\leq T_N}q^{-f(\varepsilon)k}\\
    &\geq \left(\frac{q-1}{q}-\varepsilon^{d-1}\right)\frac{\varepsilon^d(q-1)^2}{q^{2d}}\frac{q^{-\lceil T_0\rceil f(\varepsilon)}(1-q^{-(\lfloor T_N\rfloor-\lceil T_0\rceil)f(\varepsilon)})}{1-q^{-f(\varepsilon)}}\\
    &\geq \left(\frac{q-1}{q}-\varepsilon^{d-1}\right)\frac{\varepsilon^d(q-1)^2}{q^{2d}}\left(\frac{\varepsilon}{q}\right)^{\frac{d^2}{d-1}f(\varepsilon)}\frac{1-\delta}{q^{f(\varepsilon)}-1}.
    \end{split}\]
    We note that since $f(\varepsilon)\leq 1$,  $s=\frac{d^2}{d+1}+\frac{d}{d+1}f(\varepsilon)\leq d$, hence, the third line of the above equation follows.
    Now, we claim that $\left(\frac{\varepsilon}{q}\right)^{\frac{d^2}{d-1}f(\varepsilon)}\geq q^{-\frac{d^2}{d-1}}$. Indeed, it is enough to show that $f(\varepsilon)\log_q\frac{q}{\varepsilon}\leq 1$ and since $\log(1+y)\leq y$ and $\log y \leq y$ for any $y>0$, we have
    \[
    f(\varepsilon)\log_q\frac{q}{\varepsilon} \leq \frac{(1-\delta)(q-1)^2}{(\log q)^2 q^{2d-1+\frac{d^2}{d-1}}}\left(\frac{q-1}{q}-\varepsilon^{d-1}\right)\varepsilon^{d-1}\leq
    \frac{(1-\delta)(q-1)^2}{(\log q)^2 q^{2d-1+\frac{d^2}{d-1}}}\left(\frac{q-1}{2q}\right)^2\leq 1.
    \]
    Therefore, we have 
    \[
    \sum_{\mathbf{v}\in \sigma_{\varepsilon,N}(\mathbf{u})}\frac{\operatorname{diam}(B(\mathbf{v}))^s}{\operatorname{diam}(B(\mathbf{u}))^s}
    \geq \left(\frac{q-1}{q}-\varepsilon^{d-1}\right)\frac{\varepsilon^d(q-1)^2(1-\delta)}{q^{2d+\frac{d^2}{d-1}}}\frac{1}{q^{f(\varepsilon)}-1}=1.
    \]
    Since the condition \eqref{item_lower_2} of Theorem \ref{dim_HlowBnd} with $c=0$ holds, we have 
    $\dim_H\mathbf{DI}_d(\varepsilon)\geq \frac{d^2}{d+1}+\frac{d}{d+1}f(\varepsilon)$. 
    
\end{proof}
\begin{corollary}\label{Cor_singlowerbound}
    For $d\geq 2$, $\dim_H\mathbf{Sing}_d \geq \frac{d^2}{d+1}$.
\end{corollary}
\begin{proof}
    Consider two real sequences $(\varepsilon_i)_{i\geq 0}$ and $(N_i)_{i\geq 0}$ given by
    \[
    \varepsilon_i = \frac{1}{\log (i+1)} \quad\text{and}\quad N_i = i+1.\]
    Let $\mathbf{u}_{-1}\in Q$ and let $\mathbf{u}_0\in F_{N_0}(\mathbf{u}_{-1},\varepsilon_0)$. Define $Q_0=\{\mathbf{u}_0\}$ and $Q_{i+1}=\bigcup_{\mathbf{u}\in Q_i}F_{N_i}(\mathbf{u},\varepsilon_i)$ for each $i\geq 0$. For each $\mathbf{u}\in Q_i$, define $\sigma'(\mathbf{u})=F_{N_i}(\mathbf{u},\varepsilon_i)$ and $B(\mathbf{u})=B\left(\hat{\mathbf{u}},\frac{\lambda_1(\mathbf{u})}{q\vert \mathbf{u}\vert}\right)$. Finally, let $Q'=\bigcup_{i\geq 0}Q_i$.
    By the same arguments of the proof of Proposition \ref{F(u,u)FracStruc}, the triple $(Q',\sigma',B)$ is a strictly nested self-similar structure covering a subset of $\mathbf{Sing}_d$.
   
    It follows from Remark \ref{rem_lambda_bound} that $\hat{\lambda}_d(\mathbf{u})\leq \frac{q^{d-1}}{\varepsilon_{i-1}^{d-1}}<\frac{q^{d-1}}{\varepsilon_{i}^{d-1}}$ for any $\mathbf{u}\in Q_i$. By Lemma \ref{lem:B(v),B(w)Dist}, for $\mathbf{u}\in Q_i$ and any distinct $\mathbf{v},\mathbf{w}\in F_{N_i}(\mathbf{u},\varepsilon_i)$ we have $$d\left(B(\mathbf{v}),B(\mathbf{w})\right)\geq \left(\frac{\varepsilon_i}{q^{N_i+1}\hat{\lambda}_d(\mathbf{u})}\right)^{\frac{2d}{d-1}}\frac{\lambda_1(\mathbf{u})}{\vert \mathbf{u}\vert}> q\left(\frac{\varepsilon_i^{d}}{q^{N_i+d}}\right)^{\frac{2d}{d-1}}\frac{\lambda_1(\mathbf{u})}{q\vert \mathbf{u}\vert}.$$
    Hence, we define the map $\rho:Q' \to (0,1)$ by
    $$\rho(\mathbf{u})=q\left(\frac{\varepsilon_i^{d}}{q^{N_i+d}}\right)^{\frac{2d}{d-1}}\quad\text{for each }\mathbf{u}\in Q_i.$$
    Then the condition \eqref{item_wlower_1} of Theorem \ref{thm:CCLow3.6} holds with $c=0$ and the above $\rho$.
    Moreover, since the map $\rho$ decreases with $i$, the condition \eqref{item_wlower_3} of Theorem \ref{thm:CCLow3.6} follows from \eqref{eqn:lambda1(w)/wBND}.

    By applying Proposition \ref{prop:F_NSum} with $s=\frac{d^2}{d+1}$, we obtain that for each $\mathbf{u}\in Q_i$
    \[\begin{split}
    \sum_{\mathbf{v}\in \sigma'(\mathbf{u})}\frac{(\rho(\mathbf{v})\operatorname{diam}B(\mathbf{v}))^s}{(\rho(\mathbf{u})\operatorname{diam}B(\mathbf{u}))^s}
    &=\left(\frac{\varepsilon_{i+1}^{d}}{\varepsilon_i^{d} q^{N_{i+1}-N_i}}\right)^{\frac{2d}{d-1}s} \sum_{\mathbf{v}\in F_{N_i}(\mathbf{u},\varepsilon_i)} \left(\frac{\lambda_1(\mathbf{v})}{\vert \mathbf{v}\vert}\right)^s \left( \frac{\vert \mathbf{u}\vert}{\lambda_1(\mathbf{u})}\right)^s\\
    &\geq \left(\frac{\varepsilon_{i+1}^d}{\varepsilon_i^d q}\right)^{\frac{2d}{d-1}s}\left(\frac{q-1}{q}-\varepsilon_i^{d-1}\right)\frac{\varepsilon_i^d(q-1)^2}{q^{d+s}}\left(\frac{d}{d-1}N_i -2\right).
    \end{split}\]
    Since $\frac{\varepsilon_{i+1}}{\varepsilon_i}\to 1$ and $\varepsilon_i^dN_i\rightarrow \infty$ as $i\to \infty$, the last line of the above equation diverges to infinity as $i\to \infty$. Hence condition \eqref{item_wlower_2} of Theorem \ref{thm:CCLow3.6} holds for all large enough $i$. Therefore, by Theorem \ref{thm:CCLow3.6}, we have $\dim_H\mathbf{Sing}_d\geq \frac{d^2}{d+1}$.
\end{proof}
\subsection{Proof of Proposition \ref{prop:F_NSum}}\label{Subsec_proof}
By Lemma \ref{v<w}, for every $\mathbf{v}\in F_N(\mathbf{u},\varepsilon)$, $\frac{1}{q}\varepsilon\leq \hat{\lambda}_1(\mathbf{v})\leq \varepsilon$, that is, $\frac{1}{q}\varepsilon |\mathbf{v}|^{-\frac{1}{d}}\leq\lambda_1(\mathbf{v})\leq \varepsilon\vert \mathbf{v}\vert^{-\frac{1}{d}}$. Similarly, since $\mathbf{u}\in Q_{\varepsilon,N}$, we have $\frac{1}{q}\varepsilon |\mathbf{u}|^{-\frac{1}{d}}\leq\lambda_1(\mathbf{u})\leq \varepsilon\vert \mathbf{u}\vert^{-\frac{1}{d}}$. Hence, for any $s>0$, we have
\begin{equation}
\label{eqn:radSum}
    \sum_{\mathbf{v}\in F_N(\mathbf{u},\varepsilon)} \left(\frac{\lambda_1(\mathbf{v})}{\vert \mathbf{v}\vert}\right)^s \left( \frac{\vert \mathbf{u}\vert}{\lambda_1(\mathbf{u})}\right)^s
    \geq \frac{1}{q^s} \sum_{\mathbf{v}\in F_N(\mathbf{u},\varepsilon)}\left(\frac{\vert \mathbf{u}\vert}{\vert \mathbf{v}\vert}\right)^{\frac{d+1}{d}s}.
\end{equation}
As in the proofs of Propositions \ref{D(u)Cnt} and \ref{E(u,v,epsilon)Cnt}, we will divide $F_{N}(\mathbf{u},\varepsilon)$ into the following subsets $$F_{N,k}(\mathbf{u},\varepsilon)=\{\mathbf{v}\in F_{N}(\mathbf{u},\varepsilon):|\mathbf{v}|=q^k |\mathbf{u}|\}$$ for each $k\geq 1$. It is possible since $|\mathbf{v}|>|\mathbf{u}|$ for any $\mathbf{v}\in F_N(\mathbf{u},\varepsilon)$ by Lemma \ref{v<w}.  
By definition, it follows that 
\[
F_{N,k}(\mathbf{u},\varepsilon)=\bigcup_{\alpha\in \Lambda_{\mathbf{u}}(\varepsilon)\cap C_N(\mathbf{u})}\zeta_k(\mathbf{u},\alpha,\varepsilon),
\]
where
$$\zeta_k(\mathbf{u},\alpha,\varepsilon)=\Bigg\{\mathbf{v}\in Q:
\vert \mathbf{v}\vert=q^k\vert \mathbf{u}\vert,\
\pi_{\mathbf{u}}(\mathbf{v})=\alpha,\ \left(\frac{\vert \mathbf{u}\wedge \mathbf{v}\vert}{\varepsilon}\right)^{\frac{d}{d-1}}\leq \vert \mathbf{v}\vert\leq \left(\frac{q\vert \mathbf{u}\wedge \mathbf{v}\vert}{\varepsilon}\right)^{\frac{d}{d-1}}\Bigg\}.$$
If $\mathbf{v}\in \zeta_k(\mathbf{u},\alpha,\varepsilon)$, then 
\begin{equation*}
\label{eqn:qalphaIneq}
    \left(\frac{\Vert \alpha\Vert \vert \mathbf{u}\vert}{\varepsilon}\right)^{\frac{d}{d-1}}\leq \vert\mathbf{v}\vert=q^k\vert\mathbf{u}\vert\leq \left(\frac{q\Vert \alpha\Vert \vert \mathbf{u}\vert}{\varepsilon}\right)^{\frac{d}{d-1}},
\end{equation*}
which is equivalent to
\begin{equation*}
    \varepsilon q^{k\frac{d-1}{d}-1}\vert \mathbf{u}\vert^{-\frac{1}{d}}\leq \Vert \alpha\Vert\leq \varepsilon q^{k\frac{d-1}{d}}\vert \mathbf{u}\vert^{-\frac{1}{d}}.
\end{equation*}
Hence, we have 
\begin{equation}\label{Eq_F_divide}
F_{N,k}(\mathbf{u},\varepsilon)=\bigcup_{\substack{\alpha\in \Lambda_{\mathbf{u}}(\varepsilon)\cap C_N(\mathbf{u}),\\ \varepsilon q^{k\frac{d-1}{d}-1}\vert \mathbf{u}\vert^{-\frac{1}{d}}\leq \Vert \alpha\Vert\leq \varepsilon q^{k\frac{d-1}{d}}\vert \mathbf{u}\vert^{-\frac{1}{d}}}}\left\{\mathbf{v}\in \pi_{\mathbf{u}}^{-1}(\alpha)\cap Q:\vert \mathbf{v}\vert=q^k\vert \mathbf{u}\vert\right\}.
\end{equation} 
Since $\Vert \alpha\Vert\leq q^N\lambda_d(\mathbf{u})$, if $k>T_N=\frac{d}{d-1}\left(N+1+\log_q \frac{\hat{\lambda}_d(\mathbf{u})}{\varepsilon}\right)$, then there is no such $\alpha$ as in \eqref{Eq_F_divide}.
Since $\alpha\in \Lambda_\mathbf{u}(\varepsilon)$ is primitive, by Lemma \ref{lem:pi^-1(Val)}, $\#\{\mathbf{v}\in \pi_{\mathbf{u}}^{-1}(\alpha)\cap Q:\vert \mathbf{v}\vert=q^k\vert \mathbf{u}\vert\}=(q-1)q^k$. Therefore, we have
\begin{equation}\label{eqn:zeta_kSum}
\sum_{\mathbf{v}\in F_N(\mathbf{u},\varepsilon)}\left(\frac{\vert \mathbf{u}\vert}{\vert \mathbf{v}\vert}\right)^{\frac{d+1}{d}s}
=\sum_{1\leq k\leq T_N}\frac{q-1}{q^{k\left(\frac{d+1}{d}s-1\right)}}
\#\left\{\alpha\in \Lambda_{\mathbf{u}}(\varepsilon)\cap C_N(\mathbf{u}):
\frac{\varepsilon q^{k\frac{d-1}{d}-1}}{\vert \mathbf{u}\vert^{\frac{1}{d}}}\leq \Vert \alpha\Vert\leq \frac{\varepsilon q^{k\frac{d-1}{d}}}{\vert \mathbf{u}\vert^{\frac{1}{d}}}\right\}
\end{equation}
We will prove Proposition \ref{prop:F_NSum} assuming the following proposition.
\begin{proposition}
\label{Lambda_v(epsilon)capC_n'(v)Bnd}
    For $k\geq T_0 =\frac{d}{d-1}\left(1+\log_q\frac{\hat{\lambda}_d(\mathbf{u})}{\varepsilon}\right)$, we have
    \begin{equation*}
        \#\Bigg\{\alpha\in \Lambda_{\mathbf{u}}(\varepsilon)\cap C_N(\mathbf{u}):\frac{\varepsilon q^{k\frac{d-1}{d}-1}}{\vert \mathbf{u}\vert^{\frac{1}{d}}}\leq \Vert \alpha\Vert\leq \frac{\varepsilon q^{\frac{d-1}{d}k}}{\vert \mathbf{u}\vert^{\frac{1}{d}}}\Bigg\}\geq \left(\frac{q-1}{q}-\varepsilon^{d-1}\right)\frac{\varepsilon^d(q-1)}{q^d}q^{(d-1)k}.
    \end{equation*}
\end{proposition}
\begin{proof}[Proof of Proposition \ref{prop:F_NSum} assuming Proposition \ref{Lambda_v(epsilon)capC_n'(v)Bnd}]
    Combining with \eqref{eqn:radSum}, \eqref{eqn:zeta_kSum}, and Proposition \ref{Lambda_v(epsilon)capC_n'(v)Bnd}, we have
    \[
        \sum_{\mathbf{v}\in F_N(\mathbf{u},\varepsilon)} \left(\frac{\lambda_1(\mathbf{v})}{\vert \mathbf{v}\vert}\right)^s \left( \frac{\vert \mathbf{u}\vert}{\lambda_1(\mathbf{u})}\right)^s
        \geq \left(\frac{q-1}{q}-\varepsilon^{d-1}\right)\frac{\varepsilon^d(q-1)^2}{q^{d+s}}\sum_{T_0\leq k\leq T_N}q^{-(\frac{d+1}{d}s-d)k}.
    \] 
    
\end{proof}


In order to prove Proposition \ref{Lambda_v(epsilon)capC_n'(v)Bnd}, we follow the strategy of \cite[Section 8]{CC16}. 
Let $\mathbf{u}\in Q$ and $\varepsilon>0$ be given.
For any nonzero $n\in\mathcal{R}$, we will first give a lower bound of $\#(\Lambda_\mathbf{u}(\varepsilon)\cap C_n'(\mathbf{u}))$. 
By definition, we have
\[
\Lambda_\mathbf{u}(\varepsilon)\cap C_n'(\mathbf{u})
=\left\{\alpha=\sum_{i<d}m_i\xi_i+n\xi_d: 
\begin{matrix} \forall i,\ m_i\in\mathcal{R}\text{ and } |m_i|\leq \frac{|n|\lambda_d(\mathbf{u})}{\lambda_i(\mathbf{u})},\\ \alpha \text{ is primitive},\ \text{and }\hat{\lambda}_1(\alpha)>\varepsilon \end{matrix}\right\}.
\]
Since $\det(\Lambda_\alpha)=\frac{1}{|\mathbf{u}||n|\lambda_d(\mathbf{u})}$ by Lemma \ref{Lem_covol_al}, observe that 
$$\hat{\lambda}_1(\alpha)>\varepsilon \iff\lambda_1(\alpha)>\varepsilon (|\mathbf{u}||n|\lambda_d(\mathbf{u}))^{-\frac{1}{d-1}} 
\iff \lambda_1(n\Lambda_\alpha)>\varepsilon |n|(|\mathbf{u}||n|\lambda_d(\mathbf{u}))^{-\frac{1}{d-1}}.$$ 
Hence we consider the following lattice: for given $\mathbf{m}=(m_1,\dots,m_{d-1})\in \mathcal{R}^{d-1}$ and $n\in\mathcal{R}\setminus\{0\}$, we denote 
\[
L(\mathbf{m},n)=\mathcal{R}n \xi_1+\cdots+\mathcal{R}n\xi_{d-1} +\mathcal{R}\sum_{i<d}m_i \xi_i.
\]
Since $\alpha=\sum_{i<d}m_i\xi_i+n\xi_d$ is primitive if and only if $\gcd(\mathbf{m},n)=1$, we have 
\begin{equation}\label{Eq_setdescription}
\Lambda_\mathbf{u}(\varepsilon)\cap C_n'(\mathbf{u})
=\left\{\sum_{i<d}m_i\xi_i+n\xi_d: 
\begin{matrix} \forall i,\ m_i\in\mathcal{R}\text{ and } |m_i|\leq \frac{|n|\lambda_d(\mathbf{u})}{\lambda_i(\mathbf{u})},\ \gcd(\mathbf{m},n)=1,\\ \lambda_1(L(\mathbf{m},n))>\varepsilon |n|(|\mathbf{u}||n|\lambda_d(\mathbf{u}))^{-\frac{1}{d-1}} \end{matrix}\right\}.
\end{equation}
Therefore, we can estimate $\#(\Lambda_\mathbf{u}(\varepsilon)\cap C_n'(\mathbf{u}))$ by counting $\mathbf{m}\in\mathcal{R}^{d-1}$ satisfying the above conditions in \eqref{Eq_setdescription}.
Consider the set $X_n$ given by
\begin{equation*}
    X_n=\left\{\mathbf{m}\in \mathcal{R}^{d-1}:\begin{matrix} \|\mathbf{m}\|<|n|,\ \gcd(\mathbf{m},n)=1, \\
    \lambda_1(L(\mathbf{m},n))>\varepsilon |n|(|\mathbf{u}||n|\lambda_d(\mathbf{u}))^{-\frac{1}{d-1}}
    \end{matrix}\right\}.
\end{equation*}
For any $\mathbf{m}\in X_n$ and $\mathbf{k}\in\mathcal{R}^{d-1}$, observe that
$\gcd(\mathbf{m}+n\mathbf{k},n)=1$, $L(\mathbf{m}+n\mathbf{k},n)=L(\mathbf{m},n)$, and $|m_i+nk_i|\leq \frac{|n|\lambda_d(\mathbf{u})}{\lambda_i(\mathbf{u})}$ if and only if $|k_i|\leq \frac{\lambda_d(\mathbf{u})}{\lambda_i(\mathbf{u})}$ for any $i$.
Since all translates of the form $X_n +n\mathbf{k}$ with $\mathbf{k}\in \mathcal{R}^{d-1}$ are disjoint, it follows from \eqref{Eq_setdescription} that 
\begin{equation}\label{Eq_counting_trans}
\begin{split}
\#(\Lambda_\mathbf{u}(\varepsilon)\cap C_n'(\mathbf{u}))&\geq 
\# X_n \times \#\Bigg\{\mathbf{k}\in\mathcal{R}^{d-1}: \forall i,\ |k_i|\leq \frac{\lambda_d(\mathbf{u})}{\lambda_i(\mathbf{u})}\Bigg\}\\
&=\# X_n \times \frac{\lambda_d^{d-1}(\mathbf{u})}{\lambda_1(\mathbf{u})\cdots \lambda_{d-1}(\mathbf{u})}.
\end{split}
\end{equation}
Hence, it suffices to compute $\#X_n$. For this, we will follow the proof of \cite[Lemma 8.4]{CC16}.
Recall the functions $\phi$ and $D_1$ defined in \cref{subsec:BasicNT}.
\begin{proposition}\label{prop_XnCount}
    For any nonzero $n\in\mathcal{R}$, we have
    \[
    \#X_n \geq \phi(n)\left(|n|^{d-2}-\varepsilon^{d-1}D_1(n)|n|^{d-3}\right).
    \]
\end{proposition}
\begin{proof}
For simplicity, let us denote $\eta=\varepsilon|n|\left(|\mathbf{u}||n|\lambda_d(\mathbf{u})\right)^{-\frac{1}{d-1}}$. 
First, note that
\[
\lambda_1(L(\mathbf{m},n))\leq \eta \iff \widetilde{L}(\mathbf{m},n)\cap \prod_{i=1}^{d-1}B\left(0,\frac{\eta}{\lambda_i(\mathbf{u})}\right) \neq \{0\},
\]
where $\widetilde{L}(\mathbf{m},n)=\mathcal{R}n\mathbf{e}_1+\cdots+\mathcal{R}n\mathbf{e}_{d-1}+\mathcal{R}\sum_{i<d}m_i\mathbf{e}_i$. Consider the following sets:
\[\begin{split}
Y_n&=\{\mathbf{m}\in\mathcal{R}^{d-1}:\|\mathbf{m}\|<|n|\text{ and } \gcd(m_1,n)=1\};\\
Z_n&=\left\{\mathbf{m}\in Y_n : \widetilde{L}(\mathbf{m},n)\cap \prod_{i=1}^{d-1}B\left(0,\frac{\eta}{\lambda_i(\mathbf{u})}\right) \neq \{0\}\right\}.
\end{split}\]
Since $X_n \supset Y_n\setminus Z_n$ and $\#Y_n = \phi(n)|n|^{d-2}$, it is enough to show that $\#Z_n \leq \varepsilon^{d-1}\phi(n)D_1(n)|n|^{d-3}$.

Note that the lattice $\widetilde{L}(\mathbf{m},n)$ depends only on $m_1,\dots,m_{d-1}$ modulo $n$ and $\{r\in\mathcal{R}:|r|<|n|\}$ is a complete set of representatives for $\mathcal{R}/n\mathcal{R}$. Hence, by considering $B=\prod_{i=1}^{d-1}B\left(0,\frac{\eta}{\lambda_i(\mathbf{u})}\right)$ as the box in $(\mathcal{R}/n\mathcal{R})^{d-1}$, we have
\[
Z_n = \{\mathbf{m}\in (\mathcal{R}/n\mathcal{R})^{d-1}: \gcd(m_1,n)=1, \mathcal{R}\mathbf{m}\cap B\neq\{0\}\}.
\]
For $m\in\mathcal{R}/n\mathcal{R}$ with $\gcd(m,n)=1$, let
\[
Z_n(m)=\{\mathbf{r}\in(\mathcal{R}/n\mathcal{R})^{d-2}: \mathcal{R}(m,\mathbf{r})\cap B\neq\{0\}\}.
\] It follows from $\gcd(m,n)=1$, or equivalently, $m\in(\mathcal{R}/n\mathcal{R})^{*}$ that $Z_n(m)=mZ_n(1)$. Therefore, $\#Z_n=\phi(n)\#Z_n(1)$, and we need to show that $\#Z_n(1)\leq \varepsilon^{d-1}D_1(n)|n|^{d-3}$.

Since $\mathbf{r}\in Z_n(1)$ if and only if there exists $(a,\mathbf{b})\in B\setminus\{0\}$ such that $\mathbf{b}=a\mathbf{r}$, we have
\begin{equation}\label{Eq_ZnSum}
\#Z_n(1) \leq \sum_{(a,\mathbf{b})\in B\setminus\{0\}} \# E(a,\mathbf{b}),    
\end{equation}
where $E(a,\mathbf{b})=\{\mathbf{r}\in(\mathcal{R}/n\mathcal{R})^{d-2}:\mathbf{b}=a\mathbf{r}\}$. 
For $(a,\mathbf{b})\in B\setminus\{0\}$, if $a=0$ and $\mathbf{b}\neq 0$, then $E(a,\mathbf{b})$ is empty. So, we may assume that $a\neq 0$. If $\mathbf{b}\in (\mathcal{R}a)^{d-2}$, then there is $\mathbf{r}_0\in \mathcal{R}^{d-2}$ such that $\mathbf{b}=a\mathbf{r}_0$, and the general solutions to the equation $\mathbf{b}=a\mathbf{r}$ is given by 
\[
\mathbf{r}=\frac{n}{\gcd(a,n)}\mathbf{w}+\mathbf{r_0}\quad\text{for } \mathbf{w}\in \mathcal{R}^{d-2}.
\]
Therefore, the number of solutions to the equation $\mathbf{b}=a\mathbf{r}$ in $(\mathcal{R}/n\mathcal{R})^{d-2}$ is $|\gcd(a,n)|^{d-2}$ if $\mathbf{b}\in (\mathcal{R}a)^{d-2}$, and $0$ otherwise. It follows from \eqref{Eq_ZnSum} that
\[
\begin{split}
\#Z_n(1) &\leq \sum_{\substack{(a,\mathbf{b})\in B\setminus\{0\}, \\ \mathbf{b}\in(\mathcal{R}a)^{d-2}}} |\gcd(a,n)|^{d-2} 
\leq\sum_{g|n}|g|^{d-2}\sum_{\substack{a\in B(0,\frac{\eta}{\lambda_1(\mathbf{u})}) \\ g|a}} \# \left(\prod_{i=2}^{d-1}B(0,\frac{\eta}{\lambda_i(\mathbf{u})}) \cap (\mathcal{R}a)^{d-2}\right)\\
&\leq \sum_{g|n}|g|^{d-2}\sum_{\substack{a\in B(0,\frac{\eta}{\lambda_1(\mathbf{u})}) \\ g|a}} \prod_{i=2}^{d-1} \frac{\eta}{|g|\lambda_i(\mathbf{u})} = \sum_{g|n}|g|^{d-2}\prod_{i=1}^{d-1}\frac{\eta}{|g|\lambda_i(\mathbf{u})}\\
    &=\varepsilon^{d-1}|n|^{d-2}\sum_{g|n}\frac{1}{|g|}=\varepsilon^{d-1}|n|^{d-2}\sum_{g|n}|g/n|=\varepsilon^{d-1}D_1(n)|n|^{d-3}.
\end{split}
\]
This completes the proof.
\end{proof}

Combining with Lemmas \ref{lem:phiEval}, \ref{lem:D1Eval}, and Proposition \ref{prop_XnCount}, we have
\begin{lemma}\label{lem_XnSum}
For every $\ell\in \mathbb{N}$, we have
\begin{equation*}
    \sum_{\deg(n)=\ell}\#X_n\geq q^{\ell d}(q-1)\left(\frac{q-1}{q}-\varepsilon^{d-1}\right).
\end{equation*}
\end{lemma}
\label{lem:X_fSum}
\begin{proof}
    It follows from Lemmas \ref{lem:phiEval}, \ref{lem:D1Eval}, and Proposition \ref{prop_XnCount} that
    \[\begin{split}
    \sum_{\deg(n)=\ell}\#X_n &\geq \sum_{\deg(n)=\ell} \phi(n)\left(|n|^{d-2}-\varepsilon^{d-1}D_1(n)|n|^{d-3}\right) \\
    &=q^{\ell d}(q-1)\left(\frac{q-1}{q}-\varepsilon^{d-1}\right).
    \end{split}\]
\end{proof}
Now we are ready to prove Proposition \ref{Lambda_v(epsilon)capC_n'(v)Bnd}.
\begin{proof}[Proof of Proposition \ref{Lambda_v(epsilon)capC_n'(v)Bnd}]
    For $k\geq 1$, we have
\begin{alignat*}{2} 
	& \#\left\{\alpha\in \Lambda_{\mathbf{u}}(\varepsilon)\cap C_N(\mathbf{u}):
\frac{\varepsilon q^{k\frac{d-1}{d}-1}}{\vert \mathbf{u}\vert^{\frac{1}{d}}}\leq \Vert \alpha\Vert\leq \frac{\varepsilon q^{k\frac{d-1}{d}}}{\vert \mathbf{u}\vert^{\frac{1}{d}}}\right\}   \quad  \quad && \\
	& =\sum_{\frac{\varepsilon q^{k\frac{d-1}{d}-1}}{\vert \mathbf{u}\vert^{\frac{1}{d}}\lambda_d(\mathbf{u})}\leq |n|\leq \frac{\varepsilon q^{k\frac{d-1}{d}}}{\vert \mathbf{u}\vert^{\frac{1}{d}}\lambda_d(\mathbf{u})}} \#(\Lambda_{\mathbf{u}}(\varepsilon)\cap C_n'(\mathbf{u})) \quad \quad && \text{by Definition \ref{Def_lowerstructure}}  \\
 &\geq \frac{\lambda_d^{d-1}(\mathbf{u})}{\lambda_1(\mathbf{u})\cdots \lambda_{d-1}(\mathbf{u})}(q-1)\left(\frac{q-1}{q}-\varepsilon^{d-1}\right) \sum_{\ell \in I_k\cap \mathbb{Z}_{\geq 0}} q^{\ell d} \quad \quad && \text{by \eqref{Eq_counting_trans} and Lemma \ref{lem_XnSum}},
 \end{alignat*}
 where $I_k$ is the closed interval given by 
 $$I_k = \left[k\frac{d-1}{d}-1+\log_q\frac{\varepsilon}{\hat{\lambda}_d(\mathbf{u})},k\frac{d-1}{d}+\log_q\frac{\varepsilon}{\hat{\lambda}_d(\mathbf{u})}\right].$$
Note that for each $k\geq 1$ the interval $I_k$ is of length $1$. Hence, $I_k$ must contain at least one integer point. If $k\geq\frac{d}{d-1}\left(1+\log_q\frac{\hat{\lambda}_d(\mathbf{u})}{\varepsilon}\right)$, then $I_k\cap\mathbb{Z}_{\geq 0}$ is non-empty. Hence, we have
\[
\sum_{\ell\in I_k\cap\mathbb{Z}_{\geq 0}}q^{\ell d}\geq q^{(d-1)k-d}\left(\frac{\varepsilon}{\hat{\lambda}_d(\mathbf{u})}\right)^d=q^{(d-1)k}\frac{\varepsilon^d}{q^d}\frac{\lambda_1(\mathbf{u})\cdots\lambda_{d-1}(\mathbf{u})}{\lambda_d^{d-1}(\mathbf{u})}. 
\]
This completes the proof.

    
\end{proof}
\bibliography{ref}
\bibliographystyle{amsalpha}
\end{document}